\documentclass[article]{amsart}%
\usepackage{pgf,tikz}
\usepackage{amssymb}
\usepackage{amsmath}
\usepackage{amsfonts}
\usepackage{graphicx}
\usepackage{bm}%
\providecommand{\U}[1]{\protect\rule{.1in}{.1in}}
\usetikzlibrary{backgrounds}
\usetikzlibrary{arrows}
\let\oldmathbf\mathbf
\renewcommand{\mathbf}[1]{\boldsymbol{\oldmathbf{#1}}}
\newtheorem{theorem}{Theorem}

\newtheorem{corollary}[theorem]{Corollary}

\newtheorem{lemma}[theorem]{Lemma}

\allowdisplaybreaks
\begin{document}

\title{ An Euler-MacLaurin formula for polygonal sums}
\author[L. Brandolini]{Luca Brandolini}
\address{Dipartimento di Ingegneria Gestionale, dell'Informazione e della Produzione,
Universit\`a degli Studi di Bergamo, Viale Marconi 5, Dalmine BG, Italy}
\email{luca.brandolini@unibg.it}
\author[L. Colzani]{Leonardo Colzani}
\address{Dipartimento di Matematica e Applicazioni, Universit\`a di Milano-Bicocca, Via
Cozzi 55, Milano, Italy}
\email{leonardo.colzani@unimib.it}
\author[S. Robins]{Sinai Robins}
\address{Departamento de ciência da computação,
Instituto de Matemática e Estatistica,
Universidade de São Paulo,
Brasil}
\email{sinai.robins@gmail.com}
\author[G. Travaglini]{Giancarlo Travaglini}
\address{Dipartimento di Matematica e Applicazioni, Universit\`a di Milano-Bicocca, Via
Cozzi 55, Milano, Italy}
\email{giancarlo.travaglini@unimib.it}
%\subjclass{11H06, 41A55, 42B05, 65B15}
\date{}
\maketitle

\begin{abstract}
We prove an Euler-Maclaurin formula for double polygonal sums and, as a
corollary, we obtain approximate quadrature formulas for integrals of smooth
functions over polygons with integer vertices. Our Euler-Maclaurin formula is
in the spirit of Pick's theorem on the number of integer points in an integer
polygon and involves weighted Riemann sums, using tools from Harmonic analysis.
Finally, we also exhibit a classical
trick, dating back to Huygens and Newton, to accelerate convergence of these Riemann sums.

\end{abstract}%

\section{Introduction}

One motivation for this work comes from an elegant, elementary result
discovered in 1899 by G. Pick: If $P$ is a simple polygon in the Cartesian
plane with vertices with integer coordinates, if $\left\vert P\right\vert $ is
its area, and if $I$ and $B$ are the number of integer points in the interior
and on the boundary, then
\[
\left\vert P\right\vert =I+\dfrac{1}{2}B-1.
\]

There are many simple proofs of this beautiful result, see e.g. \cite{G}. In
particular, in \cite{B-C-R-T} we presented a proof based on harmonic analysis
techniques. One of the goals of this paper is to show how Pick's theorem can
be extended to more general results.  The area is an integral, and the enumeration of
integer points is a Riemann sum. Pick's theorem may therefore be thought of as
 a particular case of a
quadrature rule, as well as a particular case of an Euler-Maclaurin formula for double sums.
Such ideas have been pursued before 
(\cite{B-V}, \cite{G-P}, \cite{G-S}, \cite{K-S-W1}, \cite{K-S-W2}, \cite{K-S-W3}) but here we emphasize solid angle weights at all
integer points, thereby getting a  new  type  of  weighted  Euler-Maclaurin   summation  formula  and  quadrature  formula, with lower-order error terms.

Let $P$ be an integer polytope in $\mathbb{R}^{d}$, meaning that all vertices of $P$ have
integer coordinates. The normalized angle at a point $x$ is defined by the
proportion of $P$ in a small ball centered at $x$:
\[
\omega_{P}\left(  x\right)  =\lim_{\varepsilon\rightarrow0+}\left\{
\dfrac{\left\vert P\cap\left\{y:  \left\vert x-y\right\vert <\varepsilon
\right\}  \right\vert }{\left\vert \left\{y:  \left\vert y\right\vert
<\varepsilon\right\}  \right\vert }\right\}  .
\]
In particular, in dimension $d=1$, $P$ is a segment, $\omega_{P}\left(
x\right)  =0$ if $x$ is outside $P$, $\omega_{P}\left(  x\right)  =1$ if $x$
is inside $P$, $\omega_{P}\left(  x\right)  =1/2$ at the extremes. In
dimension $d=2$, $P$ is a polygon, $\omega_{P}\left(  x\right)  =0$ if $x$ is
outside $P$, $\omega_{P}\left(  x\right)  =1$ if $x$ is inside $P$,
$\omega_{P}\left(  x\right)  =1/2$ if $x$ is on a side but it is not a vertex,
$\omega_{P}\left(  x\right)  =\vartheta/2\pi$ if $x$ is a vertex with interior
angle $\vartheta$.

Our goal is to compare the integral of a smooth function $g\left(  x\right)  $
over  $P$,%
\[%
%TCIMACRO{\dint _{P}}%
%BeginExpansion
{\displaystyle\int_{P}}
%EndExpansion
g\left(  x\right)  dx
\]
with the weighted Riemann sums over the lattice $N^{-1}\mathbb{Z}^{d}$,
\[
N^{-d}%
%TCIMACRO{\dsum _{n\in\mathbb{Z}^{d}}}%
%BeginExpansion
{\displaystyle\sum_{n\in\mathbb{Z}^{d}}}
%EndExpansion
\omega_{P}\left(  N^{-1}n\right)  g\chi_{P}\left(  N^{-1}n\right)  .
\]
The Fourier transform of an integrable function $f\left(  x\right)  $ is
defined by%
\[
\widehat{f}\left(  \xi\right)  =%
%TCIMACRO{\dint _{\mathbb{R}^{d}}}%
%BeginExpansion
{\displaystyle\int_{\mathbb{R}^{d}}}
%EndExpansion
f\left(  x\right)  e^{-2\pi i\xi xdx}dx.
\]
Let $\varphi\left(  x\right)  $ be a radial smooth function with compact
support and integral 1 and, for $\varepsilon>0$, define $\varphi_{\varepsilon
}\left(  x\right)  =\varepsilon^{-d}\varphi\left(  \varepsilon^{-1}x\right)
$. Then one easily verify that%
\begin{align*}
&  \lim_{\varepsilon\rightarrow0+}\left\{  \varphi_{\varepsilon}\ast\left(
g\chi_{P}\right)  \left(  x\right)  \right\} \\
=  &  \lim_{\varepsilon\rightarrow0+}\left\{
%TCIMACRO{\dint _{\mathbb{R}^{d}}}%
%BeginExpansion
{\displaystyle\int_{\mathbb{R}^{d}}}
%EndExpansion
\varepsilon^{-d}\varphi\left(  \varepsilon^{-1}y\right)  g\left(  x-y\right)
\chi_{P}\left(  x-y\right)  dy\right\}  =\omega_{P}\left(  x\right)  g\left(
x\right)  .
\end{align*}
Hence, by the Poisson summation formula,%
\begin{align*}
&  N^{-d}%
%TCIMACRO{\dsum _{n\in\mathbb{Z}^{d}}}%
%BeginExpansion
{\displaystyle\sum_{n\in\mathbb{Z}^{d}}}
%EndExpansion
\omega_{P}\left(  N^{-1}n\right)  g\left(  N^{-1}n\right) \\
=  &
%TCIMACRO{\dint _{P}}%
%BeginExpansion
{\displaystyle\int_{P}}
%EndExpansion
g\left(  x\right)  dx+\lim_{\varepsilon\rightarrow0+}\left\{
%TCIMACRO{\dsum _{n\in\mathbb{Z}^{d}-\left\{  0\right\}  }}%
%BeginExpansion
{\displaystyle\sum_{n\in\mathbb{Z}^{d}-\left\{  0\right\}  }}
%EndExpansion
\widehat{\varphi}\left(  \varepsilon n\right)  \widehat{g\chi_{P}}\left(
Nn\right)  \right\}  .
\end{align*}
It should be emphasized that this application of the Poisson summation formula
and the existence of the limit are easy to prove, although not entirely trivial. 
 The function $g\left(
x\right)  \chi_{P}\left(  x\right)  $ is discontinuous when $g\left(
x\right)  \neq0$ on the boundary of $P$, hence the Fourier transform
$\widehat{g\chi_{P}}\left(  Nn\right)  $ is not absolutely integrable. The
convolution with a test function allows us to bypass the problem.

The main result of this paper is an asymptotic expansion for above weighted
Riemann sums, in dimension $d=2$,%
\[
N^{-2}%
%TCIMACRO{\dsum _{n\in\mathbb{Z}^{2}}}%
%BeginExpansion
{\displaystyle\sum_{n\in\mathbb{Z}^{2}}}
%EndExpansion
\omega_{P}\left(  N^{-1}n\right)  g\left(  N^{-1}n\right)  =%
%TCIMACRO{\dint _{P}}%
%BeginExpansion
{\displaystyle\int_{P}}
%EndExpansion
g\left(  x\right)  dx+\frac{\alpha}{N^{2}}+\frac{\beta}{N^{4}}+\frac{\gamma
}{N^{6}}+\cdots.
\]
See Theorem \ref{Theorem 3} below. Our main tools are the Poisson summation
formula together with an asympotic expansion for the Fourier transform of the
function $g\left(  x\right)  \chi_{P}\left(  x\right)  $. 

First we claim that $\widehat{g\chi_{P}%
}\left(  n\right)  $ has an asymptotic expansion.    Let us
explain this point in some more detail.   See \cite{D-L-R} for the
case $g\left(  x\right)=1$.   We offer the following outline of a possible proof.
In dimension $2$, we shall give a rigorous proof, following a different argument.

\bigskip

{\bf Claim}. {\it  If $P$ is an integer polyhedron and $g\left(  x\right)  $ is a smooth
function, then there exist functions $\left\{  \mathcal{A}_{j}\left(
n\right)  \right\}  $ homogeneous of degree $-j$, odd when $j$ is odd and even
when $j$ is even, such that the Fourier transform of $g\left(  x\right)
\chi_{P}\left(  x\right)  $ has for every $n\in Z^{d}-\left\{  0\right\}  $ an
asymptotic expansion}
\[
\widehat{g\chi_{P}}\left(  n\right)  =%
%TCIMACRO{\dsum _{j=1}^{+\infty}}%
%BeginExpansion
{\displaystyle\sum_{j=1}^{+\infty}}
%EndExpansion
\mathcal{A}_{j}\left(  n\right)  .
\]
\medskip

To see why the claim above holds, suppose that $\nabla$ is the gradient and $\mathbf{n}\left(  x\right)  $
is the outward unit normal to $\partial P$ at the point $x$, then for every
$n\neq0$ the divergence theorem gives
\begin{align*}
\widehat{\chi_{P}g}\left(  n\right)   &  =%
%TCIMACRO{\dint _{P}}%
%BeginExpansion
{\displaystyle\int_{P}}
%EndExpansion
g\left(  x\right)  e^{-2\pi in\cdot x}dx\\
&  =%
%TCIMACRO{\dint _{P}}%
%BeginExpansion
{\displaystyle\int_{P}}
%EndExpansion
\operatorname{div}\left(  \dfrac{g\left(  x\right)  e^{-2\pi in\cdot x}}{-2\pi
i\left\vert n\right\vert ^{2}}n\right)  dx-%
%TCIMACRO{\dint _{P}}%
%BeginExpansion
{\displaystyle\int_{P}}
%EndExpansion
\dfrac{n\cdot\nabla g\left(  x\right)  }{-2\pi i\left\vert n\right\vert ^{2}%
}e^{-2\pi in\cdot x}dx\\
&  =%
%TCIMACRO{\dint _{\partial P}}%
%BeginExpansion
{\displaystyle\int_{\partial P}}
%EndExpansion
\dfrac{n\cdot\mathbf{n}\left(  x\right)  g\left(  x\right)  }{-2\pi
i\left\vert n\right\vert ^{2}}e^{-2\pi in\cdot x}dx-%
%TCIMACRO{\dint _{P}}%
%BeginExpansion
{\displaystyle\int_{P}}
%EndExpansion
\dfrac{n\cdot\nabla g\left(  x\right)  }{-2\pi i\left\vert n\right\vert ^{2}%
}e^{-2\pi in\cdot x}dx.
\end{align*}
Observe that the last integral is formally analog to the first one that
defines $\widehat{g\chi_{P}}\left(  n\right)  $, but with $g\left(  x\right)
$ replaced by $\left(  n\cdot\nabla g\left(  x\right)  \right)  /\left(  2\pi
i\left\vert n\right\vert ^{2}\right)  $, which is smaller by a factor
$\left\vert n\right\vert ^{-1}$ when $\left\vert n\right\vert \rightarrow
+\infty$. Moreover, if $\left\{  F\right\}  $ are the $d-1$ dimensional faces
and $\left\{  \mathbf{n}\left(  F\right)  \right\}  $ the outward normals to
these faces, then
\[%
%TCIMACRO{\dint _{\partial P}}%
%BeginExpansion
{\displaystyle\int_{\partial P}}
%EndExpansion
\dfrac{n\cdot\mathbf{n}\left(  x\right)  g\left(  x\right)  }{-2\pi
i\left\vert n\right\vert ^{2}}e^{-2\pi in\cdot x}dx=%
%TCIMACRO{\dsum _{F}}%
%BeginExpansion
{\displaystyle\sum_{F}}
%EndExpansion
\dfrac{\mathbf{n}\left(  F\right)  \cdot n}{-2\pi i\left\vert n\right\vert
^{2}}%
%TCIMACRO{\dint _{F}}%
%BeginExpansion
{\displaystyle\int_{F}}
%EndExpansion
g\left(  x\right)  e^{-2\pi in\cdot x}dx.
\]
Fix a face $F$. Two cases are possible. If $n$ is orthogonal to the face, then
$e^{-2\pi in\cdot x}=1$ for every $x$ in the face, because of the assumption
that $n$ is an integer point and $P$ is an integer polyhedron, and the
integral over this $d-1$ dimensional face give a term homogeneous of degree
$-1$ in the asymptotic expansion. If $n$ is not orthogonal to the face, then
$e^{-2\pi in\cdot x}$ oscillates, and in a suitable coordinate system one can
apply the divergence theorem again. Again there are two possible cases. If $n$
is orthogonal to a $d-2$ dimensional face, the integral gives a term
homogeneous of degree $-2$ in the asymptotic expansion. If $n$ is not
orthogonal to a $d-2$ dimensional face, then we can keep going \ldots

\bigskip

We also conjecture that an asymptotic expansion for the Fourier transform
gives an asymptotic expansion for the weighted Riemann sums.

\bigskip

{\bf Conjecture.}  {\it  If $P$ is an integer polyhedron and $g\left(  x\right)  $ is a smooth
function, then there exist constants $\left\{  \delta\left(  h\right)
\right\}  $ such that for every integer $N$ one has an asymptotic expansion}
\[
N^{-d}%
%TCIMACRO{\dsum _{n\in\mathbb{Z}^{d}}}%
%BeginExpansion
{\displaystyle\sum_{n\in\mathbb{Z}^{d}}}
%EndExpansion
\omega_{P}\left(  N^{-1}n\right)  g\left(  N^{-1}n\right)  =%
%TCIMACRO{\dint _{P}}%
%BeginExpansion
{\displaystyle\int_{P}}
%EndExpansion
g\left(  x\right)  dx+%
%TCIMACRO{\dsum _{h=1}^{+\infty}}%
%BeginExpansion
{\displaystyle\sum_{h=1}^{+\infty}}
%EndExpansion
\dfrac{\delta\left(  h\right)  }{N^{2h}}.
\]

A possible approach to a proof is as follows.   By the Poisson summation formula, the asymptotic expansion of
$\widehat{g\chi_{P}}\left(  Nn\right)$, and the homogeneity $\mathcal{A}%
_{j}\left(  Nn\right)  =N^{-j}\mathcal{A}_{j}\left(  n\right)  $, one has
\begin{align*}
&  N^{-d}%
%TCIMACRO{\dsum _{n\in\mathbb{Z}^{d}}}%
%BeginExpansion
{\displaystyle\sum_{n\in\mathbb{Z}^{d}}}
%EndExpansion
\omega_{P}\left(  N^{-1}n\right)  g\left(  N^{-1}n\right) \\
=  &
%TCIMACRO{\dint _{P}}%
%BeginExpansion
{\displaystyle\int_{P}}
%EndExpansion
g\left(  x\right)  dx+\lim_{\varepsilon\rightarrow0+}\left\{
%TCIMACRO{\dsum _{n\in\mathbb{Z}^{d}-\left\{  0\right\}  }}%
%BeginExpansion
{\displaystyle\sum_{n\in\mathbb{Z}^{d}-\left\{  0\right\}  }}
%EndExpansion
\widehat{\varphi}\left(  \varepsilon n\right)  \widehat{g\chi_{P}}\left(
Nn\right)  \right\} \\
=  &
%TCIMACRO{\dint _{P}}%
%BeginExpansion
{\displaystyle\int_{P}}
%EndExpansion
g\left(  x\right)  dx+%
%TCIMACRO{\dsum _{j=1}^{+\infty}}%
%BeginExpansion
{\displaystyle\sum_{j=1}^{+\infty}}
%EndExpansion
N^{-j}\lim_{\varepsilon\rightarrow0+}\left\{
%TCIMACRO{\dsum _{n\in\mathbb{Z}^{d}-\left\{  0\right\}  }}%
%BeginExpansion
{\displaystyle\sum_{n\in\mathbb{Z}^{d}-\left\{  0\right\}  }}
%EndExpansion
\widehat{\varphi}\left(  \varepsilon n\right)  \mathcal{A}_{j}\left(
n\right)  \right\}  .
\end{align*}
By the symmetry $\mathcal{A}_{j}\left(  -n\right)  =\left(  -1\right)
^{j}\mathcal{A}_{j}\left(  n\right)  $, the terms of odd homogeneity sum to
zero,
\[%
%TCIMACRO{\dsum _{n\in\mathbb{Z}^{d}-\left\{  0\right\}  }}%
%BeginExpansion
{\displaystyle\sum_{n\in\mathbb{Z}^{d}-\left\{  0\right\}  }}
%EndExpansion
\widehat{\varphi}\left(  \varepsilon n\right)  \mathcal{A}_{2h+1}\left(
n\right)  =0.
\]
Moreover, the terms of even homogeneity are summable to finite limits,
\[
\lim_{\varepsilon\rightarrow0+}\left\{
%TCIMACRO{\dsum _{n\in\mathbb{Z}^{d}-\left\{  0\right\}  }}%
%BeginExpansion
{\displaystyle\sum_{n\in\mathbb{Z}^{d}-\left\{  0\right\}  }}
%EndExpansion
\widehat{\varphi}\left(  \varepsilon n\right)  \mathcal{A}_{2h}\left(
n\right)  \right\}  =\delta\left(  h\right)  .
\]
The existence of the limit is not obvious, since the limit series with 
$2h\leq  d$ are not absolutely convergent, they are only summable with the multiplier
$\widehat{\varphi}\left(  \varepsilon n\right)$.

\bigskip

See Corollary 5.3 in \cite{M-R} for a formula for $\delta\left(  1\right)  $
in the case $g\left(  x\right)  =1$. In what follows we shall show that this
conjecture is true at least in dimensions $1$ and $2$.   Our main result is Theorem \ref{Theorem 3}, 
and we include the already known theorems \ref{Theorem 1} and \ref{Theorem 2}, 
for the sake of comparison with Theorem \ref{Theorem 3}.

\section{One-dimensional Euler-Maclaurin summation}

The following Fourier analytic proof of the one dimensional Euler-Maclaurin 
summation formula is essentially the one of Poisson. See Chapter XIII in
\cite{H}. Recall that there are two possible definitions of Bernoulli
polynomials, which differ by a factorial. Here the Bernoulli polynomials in
the interval $0\leq t\leq1$ are defined recursively by
\[
B_{0}(x)=1,\ \ \ \dfrac{d}{dt}B_{j+1}(x)=B_{j}(x),\ \ \
%TCIMACRO{\dint _{0}^{1}}%
%BeginExpansion
{\displaystyle\int_{0}^{1}}
%EndExpansion
B_{j+1}(x)dx=0.
\]
In particular, $B_{0}(x)=1$, $B_{1}(x)=x-1/2$, $B_{2}(x)=x^{2}/2-x/2+1/12$,
$B_{3}(x)=x^{3}/6-x^{2}/4+x/12$,... A direct computation shows that the
periodization of $B_{1}(x)$ has the Fourier expansion%
\[
x-\left[  x\right]  -1/2=%
%TCIMACRO{\dsum \limits_{n\in\mathbb{Z}}}%
%BeginExpansion
{\displaystyle\sum\limits_{n\in\mathbb{Z}}}
%EndExpansion
\left(
%TCIMACRO{\dint _{0}^{1}}%
%BeginExpansion
{\displaystyle\int_{0}^{1}}
%EndExpansion
\left(  y-1/2\right)  e^{-2\pi iny}dy\right)  e^{2\pi inx}=-%
%TCIMACRO{\dsum \limits_{n\in\mathbb{Z-}\left\{  0\right\}  }}%
%BeginExpansion
{\displaystyle\sum\limits_{n\in\mathbb{Z-}\left\{  0\right\}  }}
%EndExpansion
\dfrac{e^{2\pi inx}}{2\pi in}.
\]
A repeated integration term by term of this series gives the Fourier
expansions of the other $B_{j}(x)$,
\[
B_{j}(x-\left[  x\right]  )=-%
%TCIMACRO{\dsum \limits_{n\in\mathbb{Z-}\left\{  0\right\}  }}%
%BeginExpansion
{\displaystyle\sum\limits_{n\in\mathbb{Z-}\left\{  0\right\}  }}
%EndExpansion
\dfrac{e^{2\pi inx}}{\left(  2\pi in\right)  ^{j}}.
\]

The following elementary computation will be used here and in the next section.

\bigskip

\begin{lemma}
\label{Lemma 1}If the function $g\left(  x\right)  $ is smooth, then for every $y \not= 0$,
\begin{align*}
&
%TCIMACRO{\dint _{a}^{b}}%
%BeginExpansion
{\displaystyle\int_{a}^{b}}
%EndExpansion
g\left(  x\right)  e^{-2\pi ixy}dx\\
=  &
%TCIMACRO{\dsum _{j=0}^{w}}%
%BeginExpansion
{\displaystyle\sum_{j=0}^{w}}
%EndExpansion
\left(  2\pi iy\right)  ^{-j-1}\left(  e^{-2\pi iay}\dfrac{d^{j}}{dx^{j}%
}g\left(  a\right)  -e^{-2\pi iby}\dfrac{d^{j}}{dx^{j}}g\left(  b\right)
\right) \\
&  +\left(  2\pi iy\right)  ^{-w-1}%
%TCIMACRO{\dint _{a}^{b}}%
%BeginExpansion
{\displaystyle\int_{a}^{b}}
%EndExpansion
\dfrac{d^{w+1}}{dx^{w+1}}g\left(  x\right)  e^{-2\pi ixy}dx.
\end{align*}

\end{lemma}

\begin{proof}
It follows by an iterated integration by parts.
\end{proof}

\bigskip

\begin{theorem}
\label{Theorem 1}If $g\left(  x\right)  $ is a smooth function on $R$, and $a$
and $b$ and $N$ are integers, then for every integer $w$,
\begin{align*}
&  \dfrac{1}{N}\left(  \dfrac{1}{2}g\left(  a\right)  +%
%TCIMACRO{\dsum _{n=Na+1}^{Nb-1}}%
%BeginExpansion
{\displaystyle\sum_{n=Na+1}^{Nb-1}}
%EndExpansion
g\left(  n/N\right)  +\dfrac{1}{2}g\left(  b\right)  \right) \\
=  &
%TCIMACRO{\dint _{a}^{b}}%
%BeginExpansion
{\displaystyle\int_{a}^{b}}
%EndExpansion
g\left(  x\right)  dx+\sum_{j=0}^{w}N^{-j-1}B_{j+1}\left(  0\right)  \left(
g^{\left(  j\right)  }\left(  b\right)  -g^{\left(  j\right)  }\left(
a\right)  \right) \\
&  +\left(  -1\right)  ^{w}N^{-w-1}%
%TCIMACRO{\dint _{a}^{b}}%
%BeginExpansion
{\displaystyle\int_{a}^{b}}
%EndExpansion
B_{w+1}\left(  Nx-\left[  Nx\right]  \right)  g^{\left(  w+1\right)  }\left(
x\right)  dx.
\end{align*}
Finally, since $B_{j+1}\left(  0\right)  =0$ when $j$ is even, in the last sum
only odd $j$ are involved.
\end{theorem}

\begin{proof}
The formula with $N\neq1$ follows from the one with $N=1$. It suffices to
replace $g\left(  x\right)  $ with $N^{-1}g\left(  N^{-1}x\right)  $ and $a$
and $b$ with $Na$ and $Nb$. Hence there is no loss of generality in assuming
$N=1$. If $\varphi\left(  t\right)  $ is a smooth even function with compact
support and integral 1, then,
\[
\lim_{\varepsilon\rightarrow0+}\left\{  \varphi_{\varepsilon}\ast\left(
g\chi_{\left[  a,b\right]  }\right)  \left(  x\right)  \right\}  =\left\{
\begin{array}
[c]{l}%
g\left(  x\right)  \ \ \ \text{if }a<x<b\text{,}\\[0.1cm]%
g\left(  x\right)  /2\ \ \ \text{if }x=a\text{ or }x=b\text{,}\\[0.1cm]%
0\ \ \ \text{if }x<a\text{ or }x>b\text{.}%
\end{array}
\right.
\]
Hence, by the Poisson summation formula,
\begin{align*}
&  \dfrac{1}{2}g\left(  a\right)  +%
%TCIMACRO{\dsum _{n=a+1}^{b-1}}%
%BeginExpansion
{\displaystyle\sum_{n=a+1}^{b-1}}
%EndExpansion
g\left(  n\right)  +\dfrac{1}{2}g\left(  b\right)  =%
%TCIMACRO{\dsum _{n=-\infty}^{+\infty}}%
%BeginExpansion
{\displaystyle\sum_{n=-\infty}^{+\infty}}
%EndExpansion
\lim_{\varepsilon\rightarrow0+}\left\{  \varphi_{\varepsilon}\ast\left(
g\chi_{\left[  a,b\right]  }\right)  \left(  n\right)  \right\} \\
=  &  \lim_{\varepsilon\rightarrow0+}\left\{
%TCIMACRO{\dsum _{n=-\infty}^{+\infty}}%
%BeginExpansion
{\displaystyle\sum_{n=-\infty}^{+\infty}}
%EndExpansion
\varphi_{\varepsilon}\ast\left(  g\chi_{\left[  a,b\right]  }\right)  \left(
n\right)  \right\}  =\lim_{\varepsilon\rightarrow0+}\left\{
%TCIMACRO{\dsum _{n=-\infty}^{+\infty}}%
%BeginExpansion
{\displaystyle\sum_{n=-\infty}^{+\infty}}
%EndExpansion
\widehat{\varphi}\left(  \varepsilon n\right)  \widehat{g\chi_{\left[
a,b\right]  }}\left(  n\right)  \right\}  .
\end{align*}
The interchange of sum and limit is justified since the convolution
$\varphi_{\varepsilon}\ast\left(  g\chi_{\left[  a,b\right]  }\right)  \left(
n\right)  $ is bounded with uniformly bounded support, hence the sum has only
a finite and bounded number of nonzero terms. And also the application of the
Poisson summation formula is legitimate, since it has been applied to a
mollification of the discontinuous function $g\left(  x\right)  \chi_{\left[
a,b\right]  }\left(  x\right)  $. Indeed all series in the above formulas are
absolutely convergent. By Lemma \ref{Lemma 1}, the Fourier transform of
$g\left(  x\right)  \chi_{\left[  a,b\right]  }\left(  x\right)  $ has the
asymptotic expansion
\begin{align*}
&  \widehat{g\chi_{\left[  a,b\right]  }}\left(  n\right)  =%
%TCIMACRO{\dint _{a}^{b}}%
%BeginExpansion
{\displaystyle\int_{a}^{b}}
%EndExpansion
g\left(  x\right)  e^{-2\pi inx}dx\\
=  &  \left\{
\begin{array}
[c]{ll}%
%TCIMACRO{\dint _{a}^{b}}%
%BeginExpansion
{\displaystyle\int_{a}^{b}}
%EndExpansion
g\left(  x\right)  dx\ \ \  & \text{if }n=0\text{,}\\[0.3cm]%
%TCIMACRO{\dsum _{j=0}^{w}}%
%BeginExpansion
{\displaystyle\sum_{j=0}^{w}}
%EndExpansion
\dfrac{g^{\left(  j\right)  }\left(  a\right)  -g^{\left(  j\right)  }\left(
b\right)  }{\left(  2\pi in\right)  ^{j+1}}+\dfrac{1}{\left(  2\pi in\right)
^{w+1}}%
%TCIMACRO{\dint _{a}^{b}}%
%BeginExpansion
{\displaystyle\int_{a}^{b}}
%EndExpansion
g^{\left(  w+1\right)  }\left(  x\right)  e^{-2\pi inx}dx & \text{if }%
n\neq0\text{.}%
\end{array}
\right.
\end{align*}
We used the assumptions that $a$, $b$, and $n$, are integers, so that
$e^{-2\pi ina}=e^{-2\pi inb}=1$. Hence,
\begin{align*}
&  \dfrac{1}{2}g\left(  a\right)  +%
%TCIMACRO{\dsum _{n=a+1}^{b-1}}%
%BeginExpansion
{\displaystyle\sum_{n=a+1}^{b-1}}
%EndExpansion
g\left(  n\right)  +\dfrac{1}{2}g\left(  b\right)  =\lim_{\varepsilon
\rightarrow0+}\left\{
%TCIMACRO{\dsum _{n\in\mathbb{Z}}}%
%BeginExpansion
{\displaystyle\sum_{n\in\mathbb{Z}}}
%EndExpansion
\widehat{\varphi}\left(  \varepsilon n\right)  \widehat{g\chi_{\left[
a,b\right]  }}\left(  n\right)  \right\} \\
=  &
%TCIMACRO{\dint _{a}^{b}}%
%BeginExpansion
{\displaystyle\int_{a}^{b}}
%EndExpansion
g\left(  t\right)  dt+\sum_{j=0}^{w}\lim_{\varepsilon\rightarrow0+}\left\{
%TCIMACRO{\dsum _{n\in\mathbb{Z}-\left\{  0\right\}  }}%
%BeginExpansion
{\displaystyle\sum_{n\in\mathbb{Z}-\left\{  0\right\}  }}
%EndExpansion
\dfrac{\widehat{\varphi}\left(  \varepsilon n\right)  }{\left(  2\pi
in\right)  ^{j+1}}\right\}  \left(  g^{\left(  j\right)  }\left(  a\right)
-g^{\left(  j\right)  }\left(  b\right)  \right) \\
&  +%
%TCIMACRO{\dint _{a}^{b}}%
%BeginExpansion
{\displaystyle\int_{a}^{b}}
%EndExpansion
\lim_{\varepsilon\rightarrow0+}\left\{
%TCIMACRO{\dsum _{n\in\mathbb{Z}-\left\{  0\right\}  }}%
%BeginExpansion
{\displaystyle\sum_{n\in\mathbb{Z}-\left\{  0\right\}  }}
%EndExpansion
\dfrac{\widehat{\varphi}\left(  \varepsilon n\right)  }{\left(  2\pi
in\right)  ^{w+1}}e^{-2\pi inx}\right\}  g^{\left(  w+1\right)  }\left(
x\right)  dx.
\end{align*}
The limits as $\varepsilon\rightarrow0+$ give the Bernoulli numbers and
polynomials,
\begin{align*}
&  \dfrac{1}{2}g\left(  a\right)  +%
%TCIMACRO{\dsum _{n=a+1}^{b-1}}%
%BeginExpansion
{\displaystyle\sum_{n=a+1}^{b-1}}
%EndExpansion
g\left(  n\right)  +\dfrac{1}{2}g\left(  b\right) \\
=  &
%TCIMACRO{\dint _{a}^{b}}%
%BeginExpansion
{\displaystyle\int_{a}^{b}}
%EndExpansion
g\left(  x\right)  dx+\sum_{j=0}^{w}B_{j+1}\left(  0\right)  \left(
g^{\left(  j\right)  }\left(  b\right)  -g^{\left(  j\right)  }\left(
a\right)  \right) \\
&  +\left(  -1\right)  ^{w}%
%TCIMACRO{\dint _{a}^{b}}%
%BeginExpansion
{\displaystyle\int_{a}^{b}}
%EndExpansion
B_{w+1}\left(  x-\left[  x\right]  \right)  g^{\left(  w+1\right)  }\left(
x\right)  dx.
\end{align*}
Finally observe that, by the symmetry of the sums that define $B_{j+1}\left(
0\right)  $, one has $B_{j+1}\left(  0\right)  =0$ when $j$ is even. Hence in
the last sum only odd $j$ are involved.
\end{proof}

\bigskip

Observe that if in the Euler-Maclaurin summation formula one disregards the
terms with $j\geq1$ and the remainder, then one obtains the trapezoidal rule
for approximating integrals,
\[
\left\vert
%TCIMACRO{\dint _{a}^{b}}%
%BeginExpansion
{\displaystyle\int_{a}^{b}}
%EndExpansion
g\left(  t\right)  dt-\dfrac{1}{N}\left(  \dfrac{1}{2}g\left(  a\right)  +%
%TCIMACRO{\dsum _{n=Na+1}^{Nb-1}}%
%BeginExpansion
{\displaystyle\sum_{n=Na+1}^{Nb-1}}
%EndExpansion
g\left(  n/N\right)  +\dfrac{1}{2}g\left(  b\right)  \right)  \right\vert
\leq\dfrac{C}{N^{2}}.
\]

\section{Two-dimensional Euler-Maclaurin summation}

A two dimensional generalization of the Euler-Maclaurin summation formula is
the following, known result.  We include this result here for the sake of comparison with
the main result, Theorem \ref{Theorem 3}.

\bigskip

\begin{theorem}
\textbf{\label{Theorem 2}}If $P$ is an open integer polygon, or a closed
integer polygon, and if $g\left(  x\right)  $ is a smooth function, then there
exist constants $\left\{  \gamma\left(  j\right)  \right\}  $ such that for
every positive integers $w$ and $N$,
\[
N^{-2}%
%TCIMACRO{\dsum _{n\in\mathbb{Z}^{2},\ N^{-1}n\in P}}%
%BeginExpansion
{\displaystyle\sum_{n\in\mathbb{Z}^{2},\ N^{-1}n\in P}}
%EndExpansion
g\left(  N^{-1}n\right)  =%
%TCIMACRO{\dint _{P}}%
%BeginExpansion
{\displaystyle\int_{P}}
%EndExpansion
g\left(  x\right)  dx+%
%TCIMACRO{\dsum _{j=1}^{w}}%
%BeginExpansion
{\displaystyle\sum_{j=1}^{w}}
%EndExpansion
\dfrac{\gamma\left(  j\right)  }{N^{j}}+\dfrac{R\left(  w,N\right)  }{N^{w+1}%
}.
\]
The remainder $R\left(  w,N\right)  $ can be bounded by a constant $C$ which does not depend on $N$.
\end{theorem}

This theorem is already known, and not only in dimension two. When $g\left(
x\right)  =1$ it is a celebrated result of Ehrhart. See Chapter 3 in
\cite{B-R}. When $g\left(  x\right)  $ is not constant, see \cite{B-V},
\cite{G-S}, \cite{K-S-W1},\cite{K-S-W2}, \cite{K-S-W3} and Chapter 12 in
\cite{B-R}. An alternative proof follows from the next theorem, the main
result of this paper, which compares the integral $%
%TCIMACRO{\dint _{P}}%
%BeginExpansion
{\displaystyle\int_{P}}
%EndExpansion
g\left(  x\right)  dx$ with the weighted Riemann sum $\frac{1}{N^{2}}%
%TCIMACRO{\dsum _{n\in\mathbb{Z}^{2}}}%
%BeginExpansion
{\displaystyle\sum_{n\in\mathbb{Z}^{2}}}
%EndExpansion
\omega_{P}\left(  N^{-1}n\right)  g\left(  N^{-1}n\right)  $.

\bigskip

\begin{theorem}    \label{Theorem 3}
If $P$ is an integer polygon in the Cartesian plane, and if
$g\left(  x\right)  $ is a smooth function, then there exist computable
constants $\left\{  \delta\left(  j\right)  \right\}  _{j=1}^{+\infty}$ with
the property that for every positive integers $w$ and $N$ there exists
$R\left(  w,N\right)  $ such that
\[
N^{-2}%
%TCIMACRO{\dsum _{n\in\mathbb{Z}^{2}}}%
%BeginExpansion
{\displaystyle\sum_{n\in\mathbb{Z}^{2}}}
%EndExpansion
\omega_{P}\left(  N^{-1}n\right)  g\left(  N^{-1}n\right)  =%
%TCIMACRO{\dint _{P}}%
%BeginExpansion
{\displaystyle\int_{P}}
%EndExpansion
g\left(  x\right)  dx+%
%TCIMACRO{\dsum _{j=1}^{w}}%
%BeginExpansion
{\displaystyle\sum_{j=1}^{w}}
%EndExpansion
\dfrac{\delta\left(  j\right)  }{N^{2j}}+\dfrac{R\left(  w,N\right)
}{N^{2w+2}}.
\]
The constants $\delta\left(  j\right)  $ depend on the derivatives
$\partial^{\left\vert \alpha\right\vert }g\left(  x\right)  /\partial
x^{\alpha}$ of order $2j-2$ and $2j-1$ evaluated at the boundary of the
polygon. The remainder $R\left(  w,N\right)  $ depends on the derivatives of
order $2w+1$ and $2w+2$\ inside the polygon. Moreover, for every $w$ there
exists $C$ such that $\left\vert R\left(  w,N\right)  \right\vert \leq C$ for
every $N$.
\end{theorem}

\bigskip

The particular case $g\left(  x\right)  =1$ of this theorem is due to
Macdonald. See Chapter 13 of \cite{B-R}. Compare the statements of Theorem
\ref{Theorem 2} and Theorem \ref{Theorem 3}:   the asymptotic expansion of the
sums without the weights $N^{-2}%
%TCIMACRO{\dsum _{n\in\mathbb{Z}^{2}}}%
%BeginExpansion
{\displaystyle\sum_{n\in\mathbb{Z}^{2}}}
%EndExpansion
g\left(  N^{-1}n\right)  $ may contain all powers of $1/N$, while the
expansion of the weighted sums
\[
N^{-2}%
%TCIMACRO{\dsum _{n\in\mathbb{Z}^{2}}}%
%BeginExpansion
{\displaystyle\sum_{n\in\mathbb{Z}^{2}}}
%EndExpansion
\omega_{P}\left(  N^{-1}n\right)  g\left(  N^{-1}n\right)
\]
contains only even powers. The unweighted sums approximate the integral $%
%TCIMACRO{\dint _{P}}%
%BeginExpansion
{\displaystyle\int_{P}}
%EndExpansion
g\left(  x\right)  dx$ to an order $1/N$, while the weighted sums give an
approximation to an order $1/N^{2}$. The weighted sums have another advantage.
While the unweighted sum of open and closed polygons are different, the
weighted sums are the same. Moreover, they are additive with respect to the
polygons. If $P$ and $Q$ are integer polygons with disjoint interiors, then
\begin{align*}
&  N^{-2}%
%TCIMACRO{\dsum _{n\in\mathbb{Z}^{2}}}%
%BeginExpansion
{\displaystyle\sum_{n\in\mathbb{Z}^{2}}}
%EndExpansion
\omega_{P\cup Q}\left(  N^{-1}n\right)  g\left(  N^{-1}n\right) \\
=  &  N^{-2}%
%TCIMACRO{\dsum _{n\in\mathbb{Z}^{2}}}%
%BeginExpansion
{\displaystyle\sum_{n\in\mathbb{Z}^{2}}}
%EndExpansion
\omega_{P}\left(  N^{-1}n\right)  g\left(  N^{-1}n\right)  +N^{-2}%
%TCIMACRO{\dsum _{n\in\mathbb{Z}^{2}}}%
%BeginExpansion
{\displaystyle\sum_{n\in\mathbb{Z}^{2}}}
%EndExpansion
\omega_{Q}\left(  N^{-1}n\right)  g\left(  N^{-1}n\right)  .
\end{align*}
In particular, the theorem for triangles implies the theorem for all other
polygons. For all these reasons, the weights $\omega_{P}\left(  N^{-1}%
n\right)  $ are quite natural.

\bigskip

\section{Proofs of Theorem \ref{Theorem 2}   and Theorem \ref{Theorem 3}}

\begin{proof}
[Proof of Theorem \ref{Theorem 2}]Theorem \ref{Theorem 2} is a corollary of
Theorem \ref{Theorem 3} and of the one dimensional Euler-Maclaurin summation
formula. The idea is the following. The difference between weighted sums in
the theorem and unweighted sums is due to the points $N^{-1}n$ on the boundary
of the polygon and, by the one dimensional Euler-Maclaurin summation formula,
the contribution of these points is asymptotically equal to $\vartheta/N+...$,
for some constant $\vartheta$,... Hence, if the asymptotic expansions of the
weighted sums contain only even powers of $1/N$, the expansions of the
unweighted sums may contain also some odd powers. The details of the proof are
as follows. Assume that $P$ is closed, and denote by $\left\{  P_{j}\right\}
$ and $\left\{  L_{j}\right\}  $ the vertices and the sided of $P$, each side
with both vertices included. The difference between weighted and unweighted
Riemann sums is due to the sampling points on the sides and the vertices of
polygon,%
\begin{align*}
&  N^{-2}%
%TCIMACRO{\dsum _{n\in\mathbb{Z}^{2},\ N^{-1}n\in P}}%
%BeginExpansion
{\displaystyle\sum_{n\in\mathbb{Z}^{2},\ N^{-1}n\in P}}
%EndExpansion
g\left(  N^{-1}n\right) \\
=  &  N^{-2}%
%TCIMACRO{\dsum _{n\in\mathbb{Z}^{2}}}%
%BeginExpansion
{\displaystyle\sum_{n\in\mathbb{Z}^{2}}}
%EndExpansion
\omega_{P}\left(  N^{-1}n\right)  g\left(  N^{-1}n\right)  +\left(  2N\right)
^{-1}%
%TCIMACRO{\dsum _{j}}%
%BeginExpansion
{\displaystyle\sum_{j}}
%EndExpansion
\left(  N^{-1}%
%TCIMACRO{\dsum _{n\in\mathbb{Z}^{2},\ N^{-1}n\in L_{j}}}%
%BeginExpansion
{\displaystyle\sum_{n\in\mathbb{Z}^{2},\ N^{-1}n\in L_{j}}}
%EndExpansion
g\left(  N^{-1}n\right)  \right) \\
&  -N^{-2}%
%TCIMACRO{\dsum _{j}}%
%BeginExpansion
{\displaystyle\sum_{j}}
%EndExpansion
\omega_{P}\left(  P_{j}\right)  g\left(  P_{j}\right)  .
\end{align*}
A similar formula holds for an open polygon. By theorem 3,
\[
N^{-2}%
%TCIMACRO{\dsum _{n\in\mathbb{Z}^{2}}}%
%BeginExpansion
{\displaystyle\sum_{n\in\mathbb{Z}^{2}}}
%EndExpansion
\omega_{P}\left(  N^{-1}n\right)  g\left(  N^{-1}n\right)  =\alpha
+\dfrac{\beta}{N^{2}}+\dfrac{\gamma}{N^{4}}+...
\]
By the one dimensional Euler-Maclaurin summation formula, for every $j$,
\[
N^{-1}%
%TCIMACRO{\dsum _{n\in\mathbb{Z}^{2},\ N^{-1}n\in L_{j}}}%
%BeginExpansion
{\displaystyle\sum_{n\in\mathbb{Z}^{2},\ N^{-1}n\in L_{j}}}
%EndExpansion
g\left(  N^{-1}n\right)  =\delta+\dfrac{\varepsilon}{N}+\dfrac{\zeta}{N^{2}%
}+...
\]
Finally,
\[
-N^{-2}%
%TCIMACRO{\dsum _{j}}%
%BeginExpansion
{\displaystyle\sum_{j}}
%EndExpansion
\omega_{P}\left(  P_{j}\right)  g\left(  P_{j}\right)  =\dfrac{\eta}{N^{2}}.
\]
Putting together these three asymptotic expansions one obtains the theorem.
\end{proof}

\bigskip

\begin{proof}
[\textbf{Proof of Theorem \ref{Theorem 3}}]The proof is in principle similar
to the one of Theorem \ref{Theorem 1}. Here it is convenient to adopt a more
explicit notation. Instead of  the vector notation  $x$  in $\mathbb{R}^{d}$ and $n$ in
$\mathbb{Z}^{d}$, we write $\left(  x,y\right)  $ in $\mathbb{R}^{2}$ and
$\left(  m,n\right)  $ in $\mathbb{Z}^{2}$. Moreover, since the dependence of
the coefficients $\left\{  \delta\left(  j\right)  \right\}  $ and the
remainder $R\left(  w,N\right)  $ from the polygon $P$ and the function
$g\left(  x,y\right)  $ is more complicated than in one variable, we keep the
parameter $N$. We denote by $\widehat{g\chi_{P}}\left(  m,n\right)  $ the
Fourier transform of $g\left(  x,y\right)  \chi_{P}\left(  x,y\right)  $,
\[
\widehat{g\chi_{P}}\left(  m,n\right)  =%
%TCIMACRO{\dint }%
%BeginExpansion
{\displaystyle\int}
%EndExpansion%
%TCIMACRO{\dint _{\mathbb{R}^{2}}}%
%BeginExpansion
{\displaystyle\int_{\mathbb{R}^{2}}}
%EndExpansion
g\left(  x,y\right)  \chi_{P}\left(  x,y\right)  e^{-2\pi i\left(
mx+ny\right)  }dxdy.
\]
Moreover, we denote by $\varphi\left(  x,y\right)  $ a smooth radial function
with compact support and integral 1. Then,%
\begin{align*}
&  \underset{\left(  m,n\right)  \in\mathbb{Z}^{2}}{%
%TCIMACRO{\dsum }%
%BeginExpansion
{\displaystyle\sum}
%EndExpansion%
%TCIMACRO{\dsum }%
%BeginExpansion
{\displaystyle\sum}
%EndExpansion
}\omega_{P}\left(  N^{-1}m,N^{-1}n\right)  N^{-2}g\left(  N^{-1}%
m,N^{-1}n\right) \\
&  =\lim_{\varepsilon\rightarrow0+}\left\{  \underset{\left(  m,n\right)
\in\mathbb{Z}^{2}}{%
%TCIMACRO{\dsum }%
%BeginExpansion
{\displaystyle\sum}
%EndExpansion%
%TCIMACRO{\dsum }%
%BeginExpansion
{\displaystyle\sum}
%EndExpansion
}\varphi_{\varepsilon}\ast\left(  g\chi_{P}\right)  _{N^{-1}}\left(
m,n\right)  \right\}  .
\end{align*}
By the Poisson summation formula,
\[
\underset{\left(  m,n\right)  \in\mathbb{Z}^{2}}{%
%TCIMACRO{\dsum }%
%BeginExpansion
{\displaystyle\sum}
%EndExpansion%
%TCIMACRO{\dsum }%
%BeginExpansion
{\displaystyle\sum}
%EndExpansion
}\varphi_{\varepsilon}\ast\left(  g\chi_{P}\right)  _{N^{-1}}\left(
m,n\right)  =\underset{\left(  m,n\right)  \in\mathbb{Z}^{2}}{%
%TCIMACRO{\dsum }%
%BeginExpansion
{\displaystyle\sum}
%EndExpansion%
%TCIMACRO{\dsum }%
%BeginExpansion
{\displaystyle\sum}
%EndExpansion
}\widehat{\varphi}\left(  \varepsilon m,\varepsilon n\right)  \widehat
{g\chi_{P}}\left(  Nm,Nn\right)  .
\]
Observe that the application of the Poisson summation formula is legitimate.
The first series is finite since both $g\left(  x,y\right)  \chi_{P}\left(
x,y\right)  $ and $\varphi\left(  x,y\right)  $ have compact support, and the
second series is absolutely convergent since $\widehat{g\chi_{P}}\left(
m,n\right)  $ is bounded and $\widehat{\varphi}\left(  m,n\right)  $ has fast
decay at infinity. Hence,
\begin{align*}
&  \underset{\left(  m,n\right)  \in\mathbb{Z}^{2}}{%
%TCIMACRO{\dsum }%
%BeginExpansion
{\displaystyle\sum}
%EndExpansion%
%TCIMACRO{\dsum }%
%BeginExpansion
{\displaystyle\sum}
%EndExpansion
}\omega_{P}\left(  N^{-1}m,N^{-1}n\right)  N^{-2}g\left(  N^{-1}%
m,N^{-1}n\right) \\
=  &
%TCIMACRO{\dint }%
%BeginExpansion
{\displaystyle\int}
%EndExpansion%
%TCIMACRO{\dint _{P}}%
%BeginExpansion
{\displaystyle\int_{P}}
%EndExpansion
g\left(  x,y\right)  dxdy+\lim_{\varepsilon\rightarrow0}\left\{
\underset{\left(  m,n\right)  \neq\left(  0,0\right)  }{%
%TCIMACRO{\dsum }%
%BeginExpansion
{\displaystyle\sum}
%EndExpansion%
%TCIMACRO{\dsum }%
%BeginExpansion
{\displaystyle\sum}
%EndExpansion
}\widehat{\varphi}\left(  \varepsilon m,\varepsilon n\right)  \widehat
{g\chi_{P}}\left(  Nm,Nn\right)  \right\}  .
\end{align*}
The Euler-Maclaurin summation formula is a quite straightforward consequence
of this version of the Poisson summation formula, and of an asymptotic
expansion of the Fourier transform $\widehat{g\chi_{P}}\left(  Nm,Nn\right)
$. Observe that both Fourier transforms and weighted Riemann sums are additive
with respect to integer polygons with disjoint interiors. Since polygons with
more than three sides have at least two ears, see \cite{M}, or simply have
interior diagonals, an integer polygon can be decomposed into integer
triangles, and since with an affine change of variables one can transform a
triangle into the simplex $\left\{  0\leq x,y,x+y\leq1\right\}  $, it suffices
to compute the asymptotic expansion of the Fourier transform of this simplex.
An iterated application of Lemma \ref{Lemma 1} gives an asymptotic expansion
of the Fourier transform of a smooth function restricted to the simplex
$\left\{  0\leq x,y,x+y\leq1\right\}  $. In this asymptotic expansion there is
a difference between directions orthogonal to the sides of the simplex, and
generic directions non orthogonal to the sides. Then the proof of Theorem
\ref{Theorem 3} follows from a few technical lemmas.
\end{proof}

\bigskip

%\section{Technical Lemmas}

We now proceed to develop some technical lemmas which will allow us to prove very precise EM-formulas
for polygons, as well as quadrature-type formulas for polygons.   We begin with the simplest right triangle in the plane.

\begin{lemma}
\textbf{\label{Lemma 2}}Assume that the function $g\left(  x,y\right)  $ is
smooth and let
\[
T=\left\{  0\leq x,y,x+y\leq1\right\}  .
\]

\noindent(1) There exist constants $\left\{  \alpha\left(  j\right)  \right\}
$, $\left\{  \beta\left(  j\right)  \right\}  $, $\left\{  \gamma\left(
j\right)  \right\}  $, such that for every non zero integer $n$ and every
$w$,
\[
\widehat{g\chi_{T}}\left(  n,0\right)  =%
%TCIMACRO{\dint _{0}^{1}}%
%BeginExpansion
{\displaystyle\int_{0}^{1}}
%EndExpansion
\left(
%TCIMACRO{\dint _{0}^{1-x}}%
%BeginExpansion
{\displaystyle\int_{0}^{1-x}}
%EndExpansion
g\left(  x,y\right)  dy\right)  e^{-2\pi inx}dx=%
%TCIMACRO{\dsum _{j=0}^{w}}%
%BeginExpansion
{\displaystyle\sum_{j=0}^{w}}
%EndExpansion
\frac{\alpha\left(  j\right)  }{n^{j+1}}+\frac{C_{1}\left(  w,n\right)
}{n^{w+1}},
\]

\[
\widehat{g\chi_{T}}\left(  0,n\right)  =%
%TCIMACRO{\dint _{0}^{1}}%
%BeginExpansion
{\displaystyle\int_{0}^{1}}
%EndExpansion
\left(
%TCIMACRO{\dint _{0}^{1-y}}%
%BeginExpansion
{\displaystyle\int_{0}^{1-y}}
%EndExpansion
g\left(  x,y\right)  dx\right)  e^{-2\pi iny}dy=%
%TCIMACRO{\dsum _{j=0}^{w}}%
%BeginExpansion
{\displaystyle\sum_{j=0}^{w}}
%EndExpansion
\frac{\beta\left(  j\right)  }{n^{j+1}}+\frac{C_{2}\left(  w,n\right)
}{n^{w+1}},
\]

\[
\widehat{g\chi_{T}}\left(  n,n\right)  =%
%TCIMACRO{\dint _{0}^{1}}%
%BeginExpansion
{\displaystyle\int_{0}^{1}}
%EndExpansion
\left(
%TCIMACRO{\dint _{0}^{1-x}}%
%BeginExpansion
{\displaystyle\int_{0}^{1-x}}
%EndExpansion
g\left(  x,y\right)  e^{-2\pi iny}dy\right)  e^{-2\pi inx}dx=%
%TCIMACRO{\dsum _{j=0}^{w}}%
%BeginExpansion
{\displaystyle\sum_{j=0}^{w}}
%EndExpansion
\frac{\gamma\left(  j\right)  }{n^{j+1}}+\frac{C_{3}\left(  w,n\right)
}{n^{w+1}}.
\]
The constants $\alpha\left(  j\right)  $, $\beta\left(  j\right)  $,
$\gamma\left(  j\right)  $ depend on the partial derivatives of $g\left(
x,y\right)  $ of order $j-1$ and $j$ evaluated on the boundary of the triangle
with vertices $\left(  0,0\right)  $, $\left(  1,0\right)  $, $\left(
0,1\right)  $, and the remainders $C_{1}\left(  w,n\right)  $, $C_{2}\left(
w,n\right)  $ , $C_{3}\left(  w,n\right)  $ depend on the partial derivatives
of $g\left(  x,y\right)  $ of order $w$ and $w+1$ in this triangle. Moreover,
for some constant $C$,
\[%
%TCIMACRO{\dsum _{n\in\mathbb{Z-}\left\{  0\right\}  }}%
%BeginExpansion
{\displaystyle\sum_{n\in\mathbb{Z-}\left\{  0\right\}  }}
%EndExpansion
\left\vert C_{1}\left(  w,n\right)  \right\vert ^{2}\leq C,\ \ \
%TCIMACRO{\dsum _{n\in\mathbb{Z-}\left\{  0\right\}  }}%
%BeginExpansion
{\displaystyle\sum_{n\in\mathbb{Z-}\left\{  0\right\}  }}
%EndExpansion
\left\vert C_{2}\left(  w,n\right)  \right\vert ^{2}\leq C,\ \ \
%TCIMACRO{\dsum _{n\in\mathbb{Z-}\left\{  0\right\}  }}%
%BeginExpansion
{\displaystyle\sum_{n\in\mathbb{Z-}\left\{  0\right\}  }}
%EndExpansion
\left\vert C_{3}\left(  w,n\right)  \right\vert ^{2}\leq C.
\]

\noindent(2) There exist constants $\left\{  \alpha\left(  h,k\right)
\right\}  $ and $\left\{  \beta\left(  h,k\right)  \right\}  $, such that for
every non zero integers $m$ and $n$, with $m\neq n$, and every $w$,
\begin{align*}
\widehat{g\chi_{T}}\left(  m,n\right)   &  =%
%TCIMACRO{\dint _{0}^{1}}%
%BeginExpansion
{\displaystyle\int_{0}^{1}}
%EndExpansion
\left(
%TCIMACRO{\dint _{0}^{1-x}}%
%BeginExpansion
{\displaystyle\int_{0}^{1-x}}
%EndExpansion
g\left(  x,y\right)  e^{-2\pi iny}dy\right)  e^{-2\pi imx}dx\\
&  =%
%TCIMACRO{\dsum _{j=0}^{w}}%
%BeginExpansion
{\displaystyle\sum_{j=0}^{w}}
%EndExpansion
\left(
%TCIMACRO{\dsum _{h+k=j}}%
%BeginExpansion
{\displaystyle\sum_{h+k=j}}
%EndExpansion
\frac{\alpha\left(  h,k\right)  }{m^{h+1}n^{k+1}}+%
%TCIMACRO{\dsum _{h+k=j}}%
%BeginExpansion
{\displaystyle\sum_{h+k=j}}
%EndExpansion
\frac{\beta\left(  h,k\right)  }{\left(  m-n\right)  ^{h+1}n^{k+1}}\right)
+R\left(  w,m,n\right)  .
\end{align*}
The constants $\alpha\left(  h,k\right)  $ and $\beta\left(  h,k\right)  $
depend on the partial derivatives of $g\left(  x,y\right)  $ of order $h+k$
evaluated at the points $\left(  0,0\right)  $, $\left(  1,0\right)  $,
$\left(  0,1\right)  $. The remainder $R\left(  w,m,n\right)  $ has the form
\begin{align*}
R\left(  w,m,n\right)   &  =%
%TCIMACRO{\dsum _{k=0}^{w}}%
%BeginExpansion
{\displaystyle\sum_{k=0}^{w}}
%EndExpansion
\frac{A_{k}\left(  w,m\right)  }{m^{w-k+1}n^{k+1}}+%
%TCIMACRO{\dsum _{k=0}^{w}}%
%BeginExpansion
{\displaystyle\sum_{k=0}^{w}}
%EndExpansion
\frac{B_{k}\left(  w,m-n\right)  }{\left(  m-n\right)  ^{w-k+1}n^{k+1}}\\
&  +\frac{\Gamma_{1}\left(  w,n\right)  }{mn^{w+1}}+\frac{\Gamma_{2}\left(
w,m-n\right)  }{mn^{w+1}}+\frac{\Gamma_{3}\left(  w,m,n\right)  }{mn^{w+1}}.
\end{align*}
The functions $A_{k}\left(  w,m\right)  $, $B_{k}\left(  w,m-n\right)  $,
$\Gamma_{1}\left(  w,n\right)  $, $\Gamma_{2}\left(  w,m-n\right)  $,
$\Gamma_{3}\left(  w,m,n\right)  $, depend on the partial derivatives of
$g\left(  x,y\right)  $ of order $w+1$ in the triangle with vertices $\left(
0,0\right)  $, $\left(  1,0\right)  $, $\left(  0,1\right)  $. Moreover, for
some constant $C$,
\begin{gather*}%
%TCIMACRO{\dsum _{m\in\mathbb{Z-}\left\{  0\right\}  }}%
%BeginExpansion
{\displaystyle\sum_{m\in\mathbb{Z-}\left\{  0\right\}  }}
%EndExpansion
\left\vert A_{k}\left(  w,m\right)  \right\vert ^{2}\leq C,\ \ \
%TCIMACRO{\dsum _{n\in\mathbb{Z-}\left\{  0\right\}  }}%
%BeginExpansion
{\displaystyle\sum_{n\in\mathbb{Z-}\left\{  0\right\}  }}
%EndExpansion
\left\vert B_{k}\left(  w,n\right)  \right\vert ^{2}\leq C,\\%
%TCIMACRO{\dsum _{n\in\mathbb{Z-}\left\{  0\right\}  }}%
%BeginExpansion
{\displaystyle\sum_{n\in\mathbb{Z-}\left\{  0\right\}  }}
%EndExpansion
\left\vert \Gamma_{1}\left(  w,n\right)  \right\vert ^{2}\leq C,\ \ \
%TCIMACRO{\dsum _{n\in\mathbb{Z-}\left\{  0\right\}  }}%
%BeginExpansion
{\displaystyle\sum_{n\in\mathbb{Z-}\left\{  0\right\}  }}
%EndExpansion
\left\vert \Gamma_{2}\left(  w,n\right)  \right\vert ^{2}\leq C,\\%
%TCIMACRO{\dsum _{m\in\mathbb{Z-}\left\{  0\right\}  }}%
%BeginExpansion
{\displaystyle\sum_{m\in\mathbb{Z-}\left\{  0\right\}  }}
%EndExpansion%
%TCIMACRO{\dsum _{n\in\mathbb{Z-}\left\{  0\right\}  }}%
%BeginExpansion
{\displaystyle\sum_{n\in\mathbb{Z-}\left\{  0\right\}  }}
%EndExpansion
\left\vert \Gamma_{3}\left(  w,m,n\right)  \right\vert ^{2}\leq C.
\end{gather*}
\end{lemma}

\begin{proof}
(1) is the asymptotic expansion of the Fourier transform in directions
orthogonal to the sides of the simplex. The first two expansions follows
directly from Lemma \ref{Lemma 1}, and the same for the third one, but after a
change of variables. Let us consider this last one,
\begin{align*}
\widehat{g\chi_{T}}\left(  n,n\right)   &  =%
%TCIMACRO{\dint _{0}^{1}}%
%BeginExpansion
{\displaystyle\int_{0}^{1}}
%EndExpansion
\left(
%TCIMACRO{\dint _{0}^{1-x}}%
%BeginExpansion
{\displaystyle\int_{0}^{1-x}}
%EndExpansion
g\left(  x,y\right)  e^{-2\pi iny}dy\right)  e^{-2\pi inx}dx\\
&  =%
%TCIMACRO{\dint _{0}^{1}}%
%BeginExpansion
{\displaystyle\int_{0}^{1}}
%EndExpansion
\left(
%TCIMACRO{\dint _{0}^{t}}%
%BeginExpansion
{\displaystyle\int_{0}^{t}}
%EndExpansion
g\left(  s,t-s\right)  ds\right)  e^{-2\pi int}dt.
\end{align*}
The constants $\gamma\left(  j\right)  $ and the remainder $C_{3}\left(
w,n\right)  $ in the asymptotic expansion of this integral can be written
explicitly in terms of the function
\[
G\left(  t\right)  =%
%TCIMACRO{\dint _{0}^{t}}%
%BeginExpansion
{\displaystyle\int_{0}^{t}}
%EndExpansion
g\left(  s,t-s\right)  ds.
\]
Indeed, by Lemma \ref{Lemma 1},
\begin{align*}
&
%TCIMACRO{\dint _{0}^{1}}%
%BeginExpansion
{\displaystyle\int_{0}^{1}}
%EndExpansion
\left(
%TCIMACRO{\dint _{0}^{t}}%
%BeginExpansion
{\displaystyle\int_{0}^{t}}
%EndExpansion
g\left(  s,t-s\right)  ds\right)  e^{-2\pi int}dt\\
=  &
%TCIMACRO{\dsum _{j=0}^{w}}%
%BeginExpansion
{\displaystyle\sum_{j=0}^{w}}
%EndExpansion
\left(  2\pi in\right)  ^{-j-1}\left(  G^{\left(  j\right)  }\left(  0\right)
-G^{\left(  j\right)  }\left(  1\right)  \right)  +\left(  2\pi in\right)
^{-w-1}%
%TCIMACRO{\dint _{0}^{1}}%
%BeginExpansion
{\displaystyle\int_{0}^{1}}
%EndExpansion
G^{\left(  w+1\right)  }\left(  t\right)  e^{-2\pi int}dt.
\end{align*}
One has
\begin{align*}
\dfrac{d^{j}}{dt^{j}}G\left(  t\right)   &  =\dfrac{d^{j}}{dt^{j}}\left(
%TCIMACRO{\dint _{0}^{t}}%
%BeginExpansion
{\displaystyle\int_{0}^{t}}
%EndExpansion
g\left(  s,t-s\right)  ds\right) \\
&  =\dfrac{d^{j-1}}{dt^{j-1}}\left(  g\left(  t,0\right)  +%
%TCIMACRO{\dint _{0}^{t}}%
%BeginExpansion
{\displaystyle\int_{0}^{t}}
%EndExpansion
\dfrac{\partial}{\partial y}g\left(  s,t-s\right)  ds\right) \\
&  =\dfrac{d^{j-2}}{dt^{j-2}}\left(  \dfrac{\partial}{\partial x}g\left(
t,0\right)  +\dfrac{\partial}{\partial y}g\left(  t,0\right)  +%
%TCIMACRO{\dint _{0}^{t}}%
%BeginExpansion
{\displaystyle\int_{0}^{t}}
%EndExpansion
\dfrac{\partial^{2}}{\partial y^{2}}g\left(  s,t-s\right)  ds\right)  =...
\end{align*}
Hence $\gamma\left(  j\right)  $ is a sum of derivatives of $g\left(
x,y\right)  $ of order $j-1$ evaluated at the points $\left(  0,0\right)  $
and $\left(  1,0\right)  $, and an integral of derivatives of $g\left(
x,y\right)  $ of order $j$ along the side between $\left(  1,0\right)  $ and
$\left(  0,1\right)  $. Similarly, the remainder $C_{3}\left(  w,n\right)  $
depend on the partial derivatives of $g\left(  x,y\right)  $ of order $w$ and
$w+1$ in the triangle with vertices $\left(  0,0\right)  $, $\left(
1,0\right)  $, $\left(  0,1\right)  $. Finally, by Bessel's inequality,%
\[%
%TCIMACRO{\dsum _{w\in\mathbb{Z-}\left\{  0\right\}  }}%
%BeginExpansion
{\displaystyle\sum_{w\in\mathbb{Z-}\left\{  0\right\}  }}
%EndExpansion
\left\vert C_{3}\left(  w,n\right)  \right\vert ^{2}\leq\left(  2\pi\right)
^{-2w-2}%
%TCIMACRO{\dint _{0}^{1}}%
%BeginExpansion
{\displaystyle\int_{0}^{1}}
%EndExpansion
\left\vert G^{\left(  w+1\right)  }\left(  t\right)  \right\vert ^{2}dt.
\]
(2) is the asymptotic expansion of the Fourier transform in generic directions
non orthogonal to the sides of the simplex. By Lemma \ref{Lemma 1}, for every
non zero integers $m$ and $n$, with $m\neq n$,
\begin{align*}
&
%TCIMACRO{\dint _{0}^{1}}%
%BeginExpansion
{\displaystyle\int_{0}^{1}}
%EndExpansion
\left(
%TCIMACRO{\dint _{0}^{1-x}}%
%BeginExpansion
{\displaystyle\int_{0}^{1-x}}
%EndExpansion
g\left(  x,y\right)  e^{-2\pi iny}dy\right)  e^{-2\pi imx}dx\\
=  &
%TCIMACRO{\dsum _{k=0}^{w}}%
%BeginExpansion
{\displaystyle\sum_{k=0}^{w}}
%EndExpansion
\left(  2\pi in\right)  ^{-k-1}%
%TCIMACRO{\dint _{0}^{1}}%
%BeginExpansion
{\displaystyle\int_{0}^{1}}
%EndExpansion
\dfrac{\partial^{k}}{\partial y^{k}}g\left(  x,0\right)  e^{-2\pi imx}dx\\
&  -%
%TCIMACRO{\dsum _{k=0}^{w}}%
%BeginExpansion
{\displaystyle\sum_{k=0}^{w}}
%EndExpansion
\left(  2\pi in\right)  ^{-k-1}%
%TCIMACRO{\dint _{0}^{1}}%
%BeginExpansion
{\displaystyle\int_{0}^{1}}
%EndExpansion
\dfrac{\partial^{k}}{\partial y^{k}}g\left(  x,1-x\right)  e^{-2\pi i\left(
m-n\right)  x}dx\\
&  +\left(  2\pi in\right)  ^{-w-1}%
%TCIMACRO{\dint _{0}^{1}}%
%BeginExpansion
{\displaystyle\int_{0}^{1}}
%EndExpansion
\left(
%TCIMACRO{\dint _{0}^{1-x}}%
%BeginExpansion
{\displaystyle\int_{0}^{1-x}}
%EndExpansion
\dfrac{\partial^{w+1}}{\partial y^{w+1}}g\left(  x,y\right)  e^{-2\pi
iny}dy\right)  e^{-2\pi imx}dx.
\end{align*}
Again by Lemma \ref{Lemma 1}, the first sum is
\begin{align*}
&
%TCIMACRO{\dsum _{k=0}^{w}}%
%BeginExpansion
{\displaystyle\sum_{k=0}^{w}}
%EndExpansion
\left(  2\pi in\right)  ^{-k-1}%
%TCIMACRO{\dint _{0}^{1}}%
%BeginExpansion
{\displaystyle\int_{0}^{1}}
%EndExpansion
\dfrac{\partial^{k}}{\partial y^{k}}g\left(  x,0\right)  e^{-2\pi imx}dx\\
=  &
%TCIMACRO{\dsum _{k=0}^{w}}%
%BeginExpansion
{\displaystyle\sum_{k=0}^{w}}
%EndExpansion%
%TCIMACRO{\dsum _{h=0}^{w-k}}%
%BeginExpansion
{\displaystyle\sum_{h=0}^{w-k}}
%EndExpansion
\left(  2\pi im\right)  ^{-h-1}\left(  2\pi in\right)  ^{-k-1}\left(
\dfrac{\partial^{h+k}}{\partial x^{h}\partial y^{k}}g\left(  0,0\right)
-\dfrac{\partial^{h+k}}{\partial x^{h}\partial y^{k}}g\left(  1,0\right)
\right) \\
&  +%
%TCIMACRO{\dsum _{k=0}^{w}}%
%BeginExpansion
{\displaystyle\sum_{k=0}^{w}}
%EndExpansion
\left(  2\pi im\right)  ^{k-w-1}\left(  2\pi in\right)  ^{-k-1}%
%TCIMACRO{\dint _{0}^{1}}%
%BeginExpansion
{\displaystyle\int_{0}^{1}}
%EndExpansion
\dfrac{\partial^{w+1}}{\partial x^{w-k+1}\partial y^{k}}g\left(  x,0\right)
e^{-2\pi imx}dx.
\end{align*}
The terms in the double sum define $\alpha\left(  h,k\right)  m^{-h-1}%
n^{-k-1}$. The terms in the last sum define the remainders $A_{k}\left(
w,m\right)  m^{k-w-1}n^{-k-1}$ and, by Bessel's inequality,
\[%
%TCIMACRO{\dsum _{m\in\mathbb{Z-}\left\{  0\right\}  }}%
%BeginExpansion
{\displaystyle\sum_{m\in\mathbb{Z-}\left\{  0\right\}  }}
%EndExpansion
\left\vert A_{k}\left(  w,m\right)  \right\vert ^{2}\leq\left(  2\pi\right)
^{-2w-4}%
%TCIMACRO{\dint _{0}^{1}}%
%BeginExpansion
{\displaystyle\int_{0}^{1}}
%EndExpansion
\left\vert \dfrac{\partial^{w+1}}{\partial x^{w-k+1}\partial y^{k}}g\left(
x,0\right)  \right\vert ^{2}dx.
\]
Similarly,
\begin{align*}
&  -%
%TCIMACRO{\dsum _{k=0}^{w}}%
%BeginExpansion
{\displaystyle\sum_{k=0}^{w}}
%EndExpansion
\left(  2\pi in\right)  ^{-k-1}%
%TCIMACRO{\dint _{0}^{1}}%
%BeginExpansion
{\displaystyle\int_{0}^{1}}
%EndExpansion
\dfrac{\partial^{k}}{\partial y^{k}}g\left(  x,1-x\right)  e^{-2\pi i\left(
m-n\right)  x}dx\\
=  &
%TCIMACRO{\dsum _{k=0}^{w}}%
%BeginExpansion
{\displaystyle\sum_{k=0}^{w}}
%EndExpansion%
%TCIMACRO{\dsum _{h=0}^{w-k}}%
%BeginExpansion
{\displaystyle\sum_{h=0}^{w-k}}
%EndExpansion
\dfrac{\left(  2\pi i\right)  ^{-w-2}}{\left(  m-n\right)  ^{w+1-k}n^{k+1}}\\
&  \times\left(  \left.  \dfrac{\partial^{h}}{\partial x^{h}}\left(
\dfrac{\partial^{k}}{\partial y^{k}}g\left(  x,1-x\right)  \right)
\right\vert _{x=1}-\left.  \dfrac{\partial^{h}}{\partial x^{h}}\left(
\dfrac{\partial^{k}}{\partial y^{k}}g\left(  x,1-x\right)  \right)
\right\vert _{x=0}\right) \\
&  -%
%TCIMACRO{\dsum _{k=0}^{w}}%
%BeginExpansion
{\displaystyle\sum_{k=0}^{w}}
%EndExpansion
\dfrac{\left(  2\pi i\right)  ^{-w-2}}{\left(  m-n\right)  ^{w+1-k}n^{k+1}}%
%TCIMACRO{\dint _{0}^{1}}%
%BeginExpansion
{\displaystyle\int_{0}^{1}}
%EndExpansion
\dfrac{\partial^{w-k+1}}{\partial x^{w-k+1}}\left(  \dfrac{\partial^{k}%
}{\partial y^{k}}g\left(  x,1-x\right)  \right)  e^{-2\pi i\left(  m-n\right)
x}dx.
\end{align*}
The terms in the double sum define $\beta\left(  h,k\right)  \left(
m-n\right)  ^{-h-1}n^{-k-1}$. The terms in the last sum define the remainders
$B_{k}\left(  w,m-n\right)  \left(  m-n\right)  ^{k-w-1}n^{-k-1}$ and, by
Bessel's inequality,
\[%
%TCIMACRO{\dsum _{n\neq0}}%
%BeginExpansion
{\displaystyle\sum_{n\neq0}}
%EndExpansion
\left\vert B_{k}\left(  w,n\right)  \right\vert ^{2}\leq\left(  2\pi\right)
^{-2w-4}%
%TCIMACRO{\dint _{0}^{1}}%
%BeginExpansion
{\displaystyle\int_{0}^{1}}
%EndExpansion
\left\vert \dfrac{\partial^{w-k+1}}{\partial x^{w-k+1}}\left(  \dfrac
{\partial^{k}}{\partial y^{k}}g\left(  x,1-x\right)  \right)  \right\vert
^{2}dx.
\]
It remains to consider%
\[
\left(  2\pi in\right)  ^{-w-1}%
%TCIMACRO{\dint _{0}^{1}}%
%BeginExpansion
{\displaystyle\int_{0}^{1}}
%EndExpansion
\left(
%TCIMACRO{\dint _{0}^{1-x}}%
%BeginExpansion
{\displaystyle\int_{0}^{1-x}}
%EndExpansion
\dfrac{\partial^{w+1}}{\partial y^{w+1}}g\left(  x,y\right)  e^{-2\pi
iny}dy\right)  e^{-2\pi imx}dx.
\]
Integrating by parts we have%
\begin{align*}
&  \left(  2\pi in\right)  ^{-w-1}%
%TCIMACRO{\dint _{0}^{1}}%
%BeginExpansion
{\displaystyle\int_{0}^{1}}
%EndExpansion
\left(
%TCIMACRO{\dint _{0}^{1-x}}%
%BeginExpansion
{\displaystyle\int_{0}^{1-x}}
%EndExpansion
\dfrac{\partial^{w+1}}{\partial y^{w+1}}g\left(  x,y\right)  e^{-2\pi
iny}dy\right)  e^{-2\pi imx}dx\\
=  &  \left(  2\pi im\right)  ^{-1}\left(  2\pi in\right)  ^{-w-1}%
%TCIMACRO{\dint _{0}^{1}}%
%BeginExpansion
{\displaystyle\int_{0}^{1}}
%EndExpansion
\dfrac{\partial^{w+1}}{\partial y^{w+1}}g\left(  0,y\right)  e^{-2\pi iny}dy\\
&  +\left(  2\pi im\right)  ^{-1}\left(  2\pi in\right)  ^{-w-1}%
%TCIMACRO{\dint _{0}^{1}}%
%BeginExpansion
{\displaystyle\int_{0}^{1}}
%EndExpansion
\dfrac{\partial}{\partial x}\left(
%TCIMACRO{\dint _{0}^{1-x}}%
%BeginExpansion
{\displaystyle\int_{0}^{1-x}}
%EndExpansion
\dfrac{\partial^{w+1}}{\partial y^{w+1}}g\left(  x,y\right)  e^{-2\pi
iny}dy\right)  e^{-2\pi imx}dx\\
=  &  \left(  2\pi im\right)  ^{-1}\left(  2\pi in\right)  ^{-w-1}%
%TCIMACRO{\dint _{0}^{1}}%
%BeginExpansion
{\displaystyle\int_{0}^{1}}
%EndExpansion
\dfrac{\partial^{w+1}}{\partial y^{w+1}}g\left(  0,y\right)  e^{-2\pi iny}dy\\
&  -\left(  2\pi im\right)  ^{-1}\left(  2\pi in\right)  ^{-w-1}%
%TCIMACRO{\dint _{0}^{1}}%
%BeginExpansion
{\displaystyle\int_{0}^{1}}
%EndExpansion
\dfrac{\partial^{w+1}}{\partial y^{w+1}}g\left(  x,1-x\right)  e^{-2\pi
i\left(  m-n\right)  x}dx\\
&  +\left(  2\pi im\right)  ^{-1}\left(  2\pi in\right)  ^{-w-1}%
%TCIMACRO{\dint _{0}^{1}}%
%BeginExpansion
{\displaystyle\int_{0}^{1}}
%EndExpansion
\left(
%TCIMACRO{\dint _{0}^{1-x}}%
%BeginExpansion
{\displaystyle\int_{0}^{1-x}}
%EndExpansion
\dfrac{\partial^{w+2}}{\partial x\partial y^{w+1}}g\left(  x,y\right)
e^{-2\pi iny}dy\right)  e^{-2\pi imx}dx.
\end{align*}
Observe that the integral at the top depends only on $\dfrac{\partial^{w+1}%
}{\partial y^{w+1}}g\left(  x,y\right)  $ in the triangle. In particular, if
this derivative vanishes, then also the sum of the three integrals at the
bottom vanishes. These three integrals define the remainders $\Gamma
_{1}\left(  w,n\right)  m^{-1}n^{-w-1}$, $\Gamma_{2}\left(  w,m-n\right)
m^{-1}n^{-w-1}$, $\Gamma_{3}\left(  w,m,n\right)  w^{-1}n^{-w-1}$. Finally, by
Bessel's inequality,
\begin{align*}%
%TCIMACRO{\dsum _{n\neq0}}%
%BeginExpansion
{\displaystyle\sum_{n\neq0}}
%EndExpansion
\left\vert \Gamma_{1}\left(  w,n\right)  \right\vert ^{2}  &  \leq\left(
2\pi\right)  ^{-2w-4}%
%TCIMACRO{\dint _{0}^{1}}%
%BeginExpansion
{\displaystyle\int_{0}^{1}}
%EndExpansion
\left\vert \dfrac{\partial^{w+1}}{\partial y^{w+1}}g\left(  0,y\right)
\right\vert ^{2}dy,\\%
%TCIMACRO{\dsum _{n\neq0}}%
%BeginExpansion
{\displaystyle\sum_{n\neq0}}
%EndExpansion
\left\vert \Gamma_{2}\left(  w,n\right)  \right\vert ^{2}  &  \leq\left(
2\pi\right)  ^{-2w-4}%
%TCIMACRO{\dint _{0}^{1}}%
%BeginExpansion
{\displaystyle\int_{0}^{1}}
%EndExpansion
\left\vert \dfrac{\partial^{w+1}}{\partial y^{w+1}}g\left(  x,1-x\right)
\right\vert ^{2}dx,\\%
%TCIMACRO{\dsum _{m\neq0}}%
%BeginExpansion
{\displaystyle\sum_{m\neq0}}
%EndExpansion%
%TCIMACRO{\dsum _{n\neq0}}%
%BeginExpansion
{\displaystyle\sum_{n\neq0}}
%EndExpansion
\left\vert \Gamma_{3}\left(  w,m,n\right)  \right\vert ^{2}  &  \leq\left(
2\pi\right)  ^{-2w-4}%
%TCIMACRO{\dint _{0}^{1}}%
%BeginExpansion
{\displaystyle\int_{0}^{1}}
%EndExpansion
\left(
%TCIMACRO{\dint _{0}^{1-x}}%
%BeginExpansion
{\displaystyle\int_{0}^{1-x}}
%EndExpansion
\left\vert \dfrac{\partial^{w+2}}{\partial x\partial y^{w+1}}g\left(
x,y\right)  \right\vert ^{2}dy\right)  dx.
\end{align*}

\end{proof}

\begin{lemma}
\label{Lemma 3}If $P$ is the triangle with vertices $\left(  p,q\right)  $,
$\left(  p+a,q+b\right)  $, $\left(  p+c,q+d\right)  $, and $T$ is the
triangle with vertices $\left(  0,0\right)  $, $\left(  1,0\right)  $,
$\left(  0,1\right)  $, then
\begin{align*}%
%TCIMACRO{\dint }%
%BeginExpansion
{\displaystyle\int}
%EndExpansion%
%TCIMACRO{\dint _{P}}%
%BeginExpansion
{\displaystyle\int_{P}}
%EndExpansion
&  g\left(  x,y\right)  e^{-2\pi i\left(  mx+ny\right)  }dxdy=e^{-2\pi
i\left(  pm+qn\right)  }\left\vert ad-bc\right\vert \\
&  \times%
%TCIMACRO{\dint }%
%BeginExpansion
{\displaystyle\int}
%EndExpansion%
%TCIMACRO{\dint _{T}}%
%BeginExpansion
{\displaystyle\int_{T}}
%EndExpansion
g\left(  p+as+ct,q+bs+dt\right)  e^{-2\pi i\left(  \left(  am+bn\right)
s+\left(  cm+dn\right)  t\right)  }dsdt.
\end{align*}

\end{lemma}

\begin{proof}
This follows by a change of variables.
\end{proof}

\bigskip

By the above lemmas, if $P$ is the triangle with vertices $\left(  p,q\right)
$, $\left(  p+a,q+b\right)  $, $\left(  p+c,q+d\right)  $, then the asymptotic
expansion of%
\[
\lim_{\varepsilon\rightarrow0+}\left\{  \underset{\left(  m,n\right)
\in\mathbb{Z}^{2}-\left\{  \left(  0,0\right)  \right\}  }{%
%TCIMACRO{\dsum }%
%BeginExpansion
{\displaystyle\sum}
%EndExpansion%
%TCIMACRO{\dsum }%
%BeginExpansion
{\displaystyle\sum}
%EndExpansion
}\widehat{\varphi}\left(  \varepsilon m,\varepsilon n\right)  \widehat
{g\chi_{P}}\left(  Nm,Nn\right)  \right\}
\]
is a sum of several terms. By part (1) of Lemma \ref{Lemma 2} there are terms
of the form%
\[
N^{-h-1}\lim_{\varepsilon\rightarrow0+}\left\{  \underset{\left(  m,n\right)
\in\mathbb{Z}^{2}-\left\{  \left(  0,0\right)  \right\}  ,\ cm+dn=0}{%
%TCIMACRO{\dsum }%
%BeginExpansion
{\displaystyle\sum}
%EndExpansion%
%TCIMACRO{\dsum }%
%BeginExpansion
{\displaystyle\sum}
%EndExpansion
}\frac{\widehat{\varphi}\left(  \varepsilon m,\varepsilon n\right)  }{\left(
am+bn\right)  ^{h+1}}\right\}  .
\]
By part (2) of Lemma \ref{Lemma 2} there are terms of the form%
\begin{gather*}
N^{-h-k-2}\lim_{\varepsilon\rightarrow0+}\left\{  \underset{am+bn\neq
0,cm+dn\neq0,\ em+fn\neq0}{%
%TCIMACRO{\dsum }%
%BeginExpansion
{\displaystyle\sum}
%EndExpansion%
%TCIMACRO{\dsum }%
%BeginExpansion
{\displaystyle\sum}
%EndExpansion
}\frac{\widehat{\varphi}\left(  \varepsilon m,\varepsilon n\right)  }{\left(
am+bn\right)  ^{h+1}\left(  cm+dn\right)  ^{k+1}}\right\} \\
=N^{-h-k-2}\lim_{\varepsilon\rightarrow0+}\left\{  \underset{am+bn\neq
0,\ cm+dn\neq0}{%
%TCIMACRO{\dsum }%
%BeginExpansion
{\displaystyle\sum}
%EndExpansion%
%TCIMACRO{\dsum }%
%BeginExpansion
{\displaystyle\sum}
%EndExpansion
}\frac{\widehat{\varphi}\left(  \varepsilon m,\varepsilon n\right)  }{\left(
am+bn\right)  ^{h+1}\left(  cm+dn\right)  ^{k+1}}\right\} \\
-N^{-h-k-2}\lim_{\varepsilon\rightarrow0+}\left\{  \underset{\left(
m,n\right)  \neq\left(  0,0\right)  ,em+fn=0}{%
%TCIMACRO{\dsum }%
%BeginExpansion
{\displaystyle\sum}
%EndExpansion%
%TCIMACRO{\dsum }%
%BeginExpansion
{\displaystyle\sum}
%EndExpansion
}\frac{\widehat{\varphi}\left(  \varepsilon m,\varepsilon n\right)  }{\left(
am+bn\right)  ^{h+1}\left(  cm+dn\right)  ^{k+1}}\right\}  .
\end{gather*}
There are also remainder terms of similar forms. $a$, $b$, $c$, $d$, $e$, $f$
are integers, and $am+bn=0$, $cm+dn=0$, $em+fn=0$ are distinct lines. In our
case $e=a-c$ and $f=b-d$. By the homogeneity of these expressions, one can
also assume that $\left(  a,b\right)  =1$, $\left(  c,d\right)  =1$, $\left(
e,f\right)  =1$. Observe the similarity of the above expansions with the
trigonometric expansions of the periodized Bernoulli polynomials:
\[
B_{k}(x-\left[  x\right]  )=\left(  \frac{-1}{2\pi i}   \right)^k
%TCIMACRO{\dsum \limits_{n\in\mathbb{Z-}\left\{  0\right\}  }}%
%BeginExpansion
{\displaystyle\sum\limits_{n\in\mathbb{Z-}\left\{  0\right\}  }}
%EndExpansion
\dfrac{e^{2\pi inx}}{n^{k}}.
\]

\begin{lemma}
\label{Lemma 4}(1) If $h$ is even,%
\[
\underset{\left(  m,n\right)  \neq\left(  0,0\right)  ,\ cm+dn=0}{%
%TCIMACRO{\dsum }%
%BeginExpansion
{\displaystyle\sum}
%EndExpansion%
%TCIMACRO{\dsum }%
%BeginExpansion
{\displaystyle\sum}
%EndExpansion
}\frac{\widehat{\varphi}\left(  \varepsilon m,\varepsilon n\right)  }{\left(
am+bn\right)  ^{h+1}}=0.
\]

\end{lemma}

(2) If $ad+bc\neq0$ with $c$ and $d$ coprime, and if $h$ is odd,
\begin{gather*}
\lim_{\varepsilon\rightarrow0+}\left\{  \underset{\left(  m,n\right)
\neq\left(  0,0\right)  ,\ cm+dn=0}{%
%TCIMACRO{\dsum }%
%BeginExpansion
{\displaystyle\sum}
%EndExpansion%
%TCIMACRO{\dsum }%
%BeginExpansion
{\displaystyle\sum}
%EndExpansion
}\frac{\widehat{\varphi}\left(  \varepsilon m,\varepsilon n\right)  }{\left(
am+bn\right)  ^{h+1}}\right\} \\
=\left(  -1\right)  ^{\left(  h-1\right)  /2}2^{h+1}\pi^{h+1}B_{h+1}\left(
0\right)  \left(  ad-bc\right)  ^{-h-1}.
\end{gather*}

(3) If $h+k$ is odd,
\[
\underset{\left(  m,n\right)  \neq\left(  0,0\right)  ,\ em+fn=0}{%
%TCIMACRO{\dsum }%
%BeginExpansion
{\displaystyle\sum}
%EndExpansion%
%TCIMACRO{\dsum }%
%BeginExpansion
{\displaystyle\sum}
%EndExpansion
}\frac{\widehat{\varphi}\left(  \varepsilon m,\varepsilon n\right)  }{\left(
am+bn\right)  ^{h+1}\left(  cm+dn\right)  ^{k+1}}=0.
\]

(4) If $af-be\neq0$, $cf-de\neq0$, with $e$ and $f$ coprime, and if $h+k$ is
even,
\begin{gather*}
\lim_{\varepsilon\rightarrow0+}\left\{  \underset{\left(  m,n\right)
\neq\left(  0,0\right)  ,\ em+fn=0}{%
%TCIMACRO{\dsum }%
%BeginExpansion
{\displaystyle\sum}
%EndExpansion%
%TCIMACRO{\dsum }%
%BeginExpansion
{\displaystyle\sum}
%EndExpansion
}\frac{\widehat{\varphi}\left(  \varepsilon m,\varepsilon n\right)  }{\left(
am+bn\right)  ^{h+1}\left(  cm+dn\right)  ^{k+1}}\right\} \\
=\frac{\left(  -1\right)  ^{\left(  h+k\right)  /2}2^{h+k+2}\pi^{h+k+2}%
B_{h+k+2}\left(  0\right)  }{\left(  af-be\right)  ^{h+1}\left(  cf-de\right)
^{k+1}}.
\end{gather*}

\begin{proof}
(1) If $h$ is even, since $\widehat{\varphi}\left(  \xi,\eta\right)  $ is
radial, then $\widehat{\varphi}\left(  \varepsilon m,\varepsilon n\right)
\left(  am+bn\right)  ^{-h-1}$ is odd and the sum vanishes. \newline(2) Assume
$h$ odd. If $\left(  c,d\right)  =1$, the non zero integer points on the line
$cm+dn=0$ are $\left(  m,n\right)  =j\left(  d,-c\right)  $, with
$j\in\mathbb{Z-}\left\{  0\right\}  $. Hence, by dominated convergence,
\begin{gather*}
\lim_{\varepsilon\rightarrow0+}\left\{  \underset{\left(  m,n\right)
\in\mathbb{Z}^{2}-\left\{  \left(  0,0\right)  \right\}  ,\ cm+dn=0}{%
%TCIMACRO{\dsum }%
%BeginExpansion
{\displaystyle\sum}
%EndExpansion%
%TCIMACRO{\dsum }%
%BeginExpansion
{\displaystyle\sum}
%EndExpansion
}\frac{\widehat{\varphi}\left(  \varepsilon m,\varepsilon n\right)  }{\left(
am+bn\right)  ^{h+1}}\right\} \\
=\left(  ad-bc\right)  ^{-h-1}%
%TCIMACRO{\dsum _{j\in\mathbb{Z-}\left\{  0\right\}  }}%
%BeginExpansion
{\displaystyle\sum_{j\in\mathbb{Z-}\left\{  0\right\}  }}
%EndExpansion
j^{-h-1}=-\left(  2\pi i\right)  ^{h+1}B_{h+1}\left(  0\right)  \left(
ad-bc\right)  ^{-h-1}.
\end{gather*}
The proof of (3) is the same as (1), and the proof of (4) is the same as (2).
\end{proof}

\bigskip

The following lemma is the analogous of the previous one for double series.

\bigskip

\begin{lemma}
\textbf{\label{Lemma 5}}Set$\ $%
\[
\mathcal{B}_{h,k}\left(  x,y\right)  =\left\{
\begin{array}
[c]{ll}%
B_{h}(x)B_{k}(y) & \ \text{if }0\leq x,y\leq1\text{,}\\
0 & \ \text{otherwise.}%
\end{array}
\right.
\]
Assume that $a$ and $b$ are coprime, that $c$ and $d$ are coprime, and that
$ad-bc\neq0$, and denote by $R$ the parallelogram
\[
R=\left\{  \left(  x,y\right)  \in\mathbb{R}^{2},\ 0\leq\dfrac{dx-cy}%
{ad-bc}\leq1,\ 0\leq\dfrac{-bx+ay}{ad-bc}\leq1\right\}
\]
(1) If $h+k$ is odd,
\[
\underset{am+bn\neq0,\ cm+dn\neq0}{%
%TCIMACRO{\dsum }%
%BeginExpansion
{\displaystyle\sum}
%EndExpansion%
%TCIMACRO{\dsum }%
%BeginExpansion
{\displaystyle\sum}
%EndExpansion
}\widehat{\varphi}\left(  \varepsilon m,\varepsilon n\right)  \left(
am+bn\right)  ^{-h-1}\left(  cm+dn\right)  ^{-k-1}=0.
\]
(2) If $h+k$ is even,%
\begin{align*}
&  \lim_{\varepsilon\rightarrow0+}\left\{  \underset{am+bn\neq0,\ cm+dn\neq0}{%
%TCIMACRO{\dsum }%
%BeginExpansion
{\displaystyle\sum}
%EndExpansion%
%TCIMACRO{\dsum }%
%BeginExpansion
{\displaystyle\sum}
%EndExpansion
}\widehat{\varphi}\left(  \varepsilon m,\varepsilon n\right)  \left(
am+bn\right)  ^{-h-1}\left(  cm+dn\right)  ^{-k-1}\right\} \\
=  &  \left(  -1\right)  ^{\left(  h+k+2\right)  /2}2^{h+k+2}\pi
^{h+k+2}\left\vert ad-bc\right\vert ^{-1}\\
&  \times\underset{\left(  m,n\right)  \in\mathbb{Z}^{2}}{%
%TCIMACRO{\dsum }%
%BeginExpansion
{\displaystyle\sum}
%EndExpansion%
%TCIMACRO{\dsum }%
%BeginExpansion
{\displaystyle\sum}
%EndExpansion
}\omega_{R}\left(  m,n\right)  \mathcal{B}_{h+1,k+1}\left(  \dfrac
{dm-cn}{ad-bc},\dfrac{-bm+an}{ad-bc}\right)  .
\end{align*}

\end{lemma}

\begin{proof}
(1) If $h+k$ is odd, then the sum vanishes by symmetry.\newline(2) Assume
$h+k$ even. The Fourier transform of
\[
\mathcal{B}_{h,k}\left(  \dfrac{dx-cy}{ad-bc},\dfrac{-bx+ay}{ad-bc}\right)
\]
evaluated at the integers $\left(  m,n\right)  $, with $am+bn\neq0$ and
$cm+dn\neq0$, is%
\begin{align*}
&
%TCIMACRO{\dint }%
%BeginExpansion
{\displaystyle\int}
%EndExpansion%
%TCIMACRO{\dint _{\mathbb{R}^{2}}}%
%BeginExpansion
{\displaystyle\int_{\mathbb{R}^{2}}}
%EndExpansion
\mathcal{B}_{h,k}\left(  \dfrac{dx-cy}{ad-bc},\dfrac{-bx+ay}{ad-bc}\right)
e^{-2\pi i\left(  mx+ny\right)  }dxdy\\
=  &  \left\vert ad-bc\right\vert
%TCIMACRO{\dint }%
%BeginExpansion
{\displaystyle\int}
%EndExpansion%
%TCIMACRO{\dint _{\mathbb{R}^{2}}}%
%BeginExpansion
{\displaystyle\int_{\mathbb{R}^{2}}}
%EndExpansion
\mathcal{B}_{h,k}\left(  s,t\right)  e^{-2\pi i\left(  m\left(  as+ct\right)
+n\left(  bs+dt\right)  \right)  }dsdt\\
=  &  \left\vert ad-bc\right\vert \left(
%TCIMACRO{\dint _{0}^{1}}%
%BeginExpansion
{\displaystyle\int_{0}^{1}}
%EndExpansion
B_{h}(s)e^{-2\pi i\left(  am+bn\right)  s}ds\right)  \left(
%TCIMACRO{\dint _{0}^{1}}%
%BeginExpansion
{\displaystyle\int_{0}^{1}}
%EndExpansion
B_{k}(t)e^{-2\pi i\left(  cm+dn\right)  t}dt\right) \\
=  &  \left\vert ad-bc\right\vert \left(  2\pi i\left(  am+bn\right)  \right)
^{-h-1}\left(  2\pi i\left(  cm+dn\right)  \right)  ^{-k-1}.
\end{align*}
Then%
\[
\underset{am+bn\neq0,\ cm+dn\neq0}{%
%TCIMACRO{\dsum }%
%BeginExpansion
{\displaystyle\sum}
%EndExpansion%
%TCIMACRO{\dsum }%
%BeginExpansion
{\displaystyle\sum}
%EndExpansion
}\left\vert ad-bc\right\vert \left(  2\pi i\left(  am+bn\right)  \right)
^{-h-1}\left(  2\pi i\left(  cm+dn\right)  \right)  ^{-k-1}%
\]
is formally the sum over $\mathbb{Z}^{2}$ of the Fourier coefficients of the
function%
\[
\mathcal{B}_{h,k}\left(  \dfrac{dx-cy}{ad-bc},\dfrac{-bx+ay}{ad-bc}\right)  .
\]
Hence (2) follows from the Poisson summation formula. The factors $\omega
_{R}\left(  m,n\right)  $ in the formula come from the fact that the function
$\mathcal{B}_{h,k}\left(  \dfrac{dx-cy}{ad-bc},\dfrac{-bx+ay}{ad-bc}\right)  $
may be discontinuous on the boundary of $R$. Observe that the sum%
\[
\underset{\left(  m,n\right)  \in\mathbb{Z}^{2}}{%
%TCIMACRO{\dsum }%
%BeginExpansion
{\displaystyle\sum}
%EndExpansion%
%TCIMACRO{\dsum }%
%BeginExpansion
{\displaystyle\sum}
%EndExpansion
}\omega_{R}\left(  m,n\right)  \mathcal{B}_{h,k}\left(  \dfrac{dm-cn}%
{ad-bc},\dfrac{-bm+an}{ad-bc}\right)
\]
is finite. Hence limit of the series%
\[
\lim_{\varepsilon\rightarrow0+}\left\{  \underset{am+bn\neq0,\ cm+dn\neq0}{%
%TCIMACRO{\dsum }%
%BeginExpansion
{\displaystyle\sum}
%EndExpansion%
%TCIMACRO{\dsum }%
%BeginExpansion
{\displaystyle\sum}
%EndExpansion
}\widehat{\varphi}\left(  \varepsilon m,\varepsilon n\right)  \left(
am+bn\right)  ^{-h-1}\left(  cm+dn\right)  ^{-k-1}\right\}
\]
can be computed explicitly in a finite number of steps.
\end{proof}

\bigskip

It remains to estimate the remainders.

\bigskip

\begin{lemma}
\textbf{\label{Lemma 6}}Assume $w>0$ and let $am+bn=0$, $cm+dn=0$, $em+fn=0$
be distinct lines.\newline(1) If $n^{-w-1}C\left(  w,n\right)  $ is one of the
remainders in part (1) of Lemma \ref{Lemma 2}, then for some constant $C$
independent of $\varepsilon$ and $N$,
\[
\left\vert \underset{\left(  m,n\right)  \neq\left(  0,0\right)  ,\ cm+dn=0}{%
%TCIMACRO{\dsum }%
%BeginExpansion
{\displaystyle\sum}
%EndExpansion%
%TCIMACRO{\dsum }%
%BeginExpansion
{\displaystyle\sum}
%EndExpansion
}\widehat{\varphi}\left(  \varepsilon m,\varepsilon n\right)  \frac{C\left(
w,N\left(  am+bn\right)  \right)  }{\left(  N\left(  am+bn\right)  \right)
^{w+1}}\right\vert \leq CN^{-w-1}.
\]
(2) If $R\left(  w,m,n\right)  $ is the remainder in part (2) of Lemma
\ref{Lemma 2}, then for some constant $C$ independent of $\varepsilon$ and
$N$,
\begin{gather*}
\left\vert \underset{am+bn\neq0,\ cm+dn\neq0,\ em+fn\neq0}{%
%TCIMACRO{\dsum }%
%BeginExpansion
{\displaystyle\sum}
%EndExpansion%
%TCIMACRO{\dsum }%
%BeginExpansion
{\displaystyle\sum}
%EndExpansion
}\widehat{\varphi}\left(  \varepsilon m,\varepsilon n\right)  R\left(
w,N\left(  am+bn\right)  ,N\left(  cm+dn\right)  \right)  \right\vert \\
\leq CN^{-w-2}.
\end{gather*}

\end{lemma}

\begin{proof}
(1) follows from the fact that $C\left(  w,N\left(  am+bn\right)  \right)  $
is bounded. The proof of (2) is a bit more complicated, since $R\left(
w,m,n\right)  $ is sum of many terms, and not all the series involved are
absolutely convergent. We shall consider just two of them. A term in $R\left(
w,m,n\right)  $ is%
\[
A_{0}\left(  w,N\left(  am+bn\right)  \right)  \left(  N\left(  am+bn\right)
\right)  ^{-w-1}\left(  N\left(  cm+dn\right)  \right)  ^{-1}.
\]
After factorizing $N^{-w-2}$, one has to estimate the series
\[
\underset{am+bn\neq0,\ cm+dn\neq0}{%
%TCIMACRO{\dsum }%
%BeginExpansion
{\displaystyle\sum}
%EndExpansion%
%TCIMACRO{\dsum }%
%BeginExpansion
{\displaystyle\sum}
%EndExpansion
}\widehat{\varphi}\left(  \varepsilon m,\varepsilon n\right)  A_{0}\left(
w,N\left(  am+bn\right)  \right)  \left(  am+bn\right)  ^{-w-1}\left(
cm+dn\right)  ^{-1}.
\]
Observe that the corresponding series without the cutoff $\widehat{\varphi
}\left(  \varepsilon m,\varepsilon n\right)  $ does not converge absolutely.
As in the previous lemma define%
\[
\mathcal{F}_{N}\left(  x,y\right)  =\left\{
\begin{array}
[c]{ll}%
%TCIMACRO{\dsum \limits_{m\neq0}}%
%BeginExpansion
{\displaystyle\sum\limits_{m\neq0}}
%EndExpansion
\dfrac{A_{0}\left(  w,Nm\right)  e^{2\pi imx}}{m^{w+1}}\cdot%
%TCIMACRO{\dsum \limits_{n\neq0}}%
%BeginExpansion
{\displaystyle\sum\limits_{n\neq0}}
%EndExpansion
\dfrac{e^{2\pi iny}}{n} & \text{if }0\leq x,y\leq1\text{,}\\
0 & \text{otherwise.}%
\end{array}
\right.
\]
The first series converges absolutely and it can be bounded independently on
$N$. And also the second series defines a bounded function, which is up to a
constant the degree one Bernoulli polynomial. Hence $\mathcal{F}_{N}\left(
x,y\right)  $ is a bounded function. Moreover, as in Lemma \ref{Lemma 5}, the
Fourier transform of $\mathcal{F}_{N}\left(  \dfrac{dx-cy}{ad-bc}%
,\dfrac{-bx+ay}{ad-bc}\right)  $ evaluated at the integers is
\begin{align*}
&
%TCIMACRO{\dint }%
%BeginExpansion
{\displaystyle\int}
%EndExpansion%
%TCIMACRO{\dint _{\mathbb{R}^{2}}}%
%BeginExpansion
{\displaystyle\int_{\mathbb{R}^{2}}}
%EndExpansion
\mathcal{F}_{N}\left(  \dfrac{dx-cy}{ad-bc},\dfrac{-bx+ay}{ad-bc}\right)
e^{-2\pi i\left(  mx+ny\right)  }dxdy\\
=  &  \left\vert ad-bc\right\vert
%TCIMACRO{\dint }%
%BeginExpansion
{\displaystyle\int}
%EndExpansion%
%TCIMACRO{\dint _{\mathbb{R}^{2}}}%
%BeginExpansion
{\displaystyle\int_{\mathbb{R}^{2}}}
%EndExpansion
\mathcal{F}_{N}\left(  s,t\right)  e^{-2\pi i\left(  m\left(  as+ct\right)
+n\left(  bs+dt\right)  \right)  }dsdt\\
=  &  \left\vert ad-bc\right\vert A_{0}\left(  w,N\left(  am+bn\right)
\right)  \left(  am+bn\right)  ^{-w-1}\left(  cm+dn\right)  ^{-1}.
\end{align*}
Then, by the Poisson summation formula,%
\begin{align*}
&  \left\vert ad-bc\right\vert ^{-1}\underset{\left(  m,n\right)
\in\mathbb{Z}^{2}}{%
%TCIMACRO{\dsum }%
%BeginExpansion
{\displaystyle\sum}
%EndExpansion%
%TCIMACRO{\dsum }%
%BeginExpansion
{\displaystyle\sum}
%EndExpansion
}%
%TCIMACRO{\dint }%
%BeginExpansion
{\displaystyle\int}
%EndExpansion%
%TCIMACRO{\dint _{\mathbb{R}^{2}}}%
%BeginExpansion
{\displaystyle\int_{\mathbb{R}^{2}}}
%EndExpansion
\varphi_{\varepsilon}\left(  x,y\right)  \mathcal{F}_{N}\left(  \dfrac
{dm-cn}{ad-bc}-x,\dfrac{-bm+an}{ad-bc}-y\right)  dxdy\\
=  &  \underset{\left(  m,n\right)  \in\mathbb{Z}^{2}}{%
%TCIMACRO{\dsum }%
%BeginExpansion
{\displaystyle\sum}
%EndExpansion%
%TCIMACRO{\dsum }%
%BeginExpansion
{\displaystyle\sum}
%EndExpansion
}\widehat{\varphi}\left(  \varepsilon m,\varepsilon n\right)  \widehat
{\mathcal{F}_{N}}\left(  am+bn,cm+dn\right) \\
=  &  \underset{\ am+bn\neq0,\ cm+dn\neq0}{%
%TCIMACRO{\dsum }%
%BeginExpansion
{\displaystyle\sum}
%EndExpansion%
%TCIMACRO{\dsum }%
%BeginExpansion
{\displaystyle\sum}
%EndExpansion
}\widehat{\varphi}\left(  \varepsilon m,\varepsilon n\right)  \frac
{A_{0}\left(  w,N\left(  am+bn\right)  \right)  }{\left(  am+bn\right)
^{w+1}\left(  cm+dn\right)  }.
\end{align*}
The first series is indeed a finite sum that can be bounded independently of
$\varepsilon$ and $N$. Hence also the last series is bounded independently of
$\varepsilon$ and $N$.\newline Another term in $R\left(  w,m,n\right)  $ is
\[
A_{w}\left(  w,N\left(  am+bn\right)  \right)  \left(  N\left(  am+bn\right)
\right)  ^{-1}\left(  N\left(  cm+dn\right)  \right)  ^{-w-1}.
\]
After factorizing $N^{-w-2}$ one has to estimate the series
\[
\underset{am+bn\neq0,\ cm+dn\neq0}{%
%TCIMACRO{\dsum }%
%BeginExpansion
{\displaystyle\sum}
%EndExpansion%
%TCIMACRO{\dsum }%
%BeginExpansion
{\displaystyle\sum}
%EndExpansion
}\widehat{\varphi}\left(  \varepsilon m,\varepsilon n\right)  \frac
{A_{w}\left(  w,N\left(  am+bn\right)  \right)  }{\left(  am+bn\right)
\left(  cm+dn\right)  ^{w+1}}.
\]
This series is absolutely convergent and it can be bounded independently of
$\varepsilon$ and $N$,
\begin{align*}
&  \underset{am+bn\neq0,\ cm+dn\neq0}{%
%TCIMACRO{\dsum }%
%BeginExpansion
{\displaystyle\sum}
%EndExpansion%
%TCIMACRO{\dsum }%
%BeginExpansion
{\displaystyle\sum}
%EndExpansion
}\left\vert \widehat{\varphi}\left(  \varepsilon m,\varepsilon n\right)
\frac{A_{w}\left(  w,N\left(  am+bn\right)  \right)  }{\left(  am+bn\right)
\left(  cm+dn\right)  ^{w+1}}\right\vert \\
\leq &  \sup_{\left(  \xi,\eta\right)  \in\mathbb{R}^{2}}\left\{  \left\vert
\widehat{\varphi}\left(  \xi,\eta\right)  \right\vert \right\}
%TCIMACRO{\dsum _{j\neq0}}%
%BeginExpansion
{\displaystyle\sum_{j\neq0}}
%EndExpansion
\left\vert j\right\vert ^{-w-1}\left\{  \underset{am+bn\neq0,\ cm+dn=j}{%
%TCIMACRO{\dsum }%
%BeginExpansion
{\displaystyle\sum}
%EndExpansion%
%TCIMACRO{\dsum }%
%BeginExpansion
{\displaystyle\sum}
%EndExpansion
}\left\vert am+bn\right\vert ^{-2}\right\}  ^{1/2}\\
&  \times\left\{  \underset{am+bn\neq0,\ cm+dn=j}{%
%TCIMACRO{\dsum }%
%BeginExpansion
{\displaystyle\sum}
%EndExpansion%
%TCIMACRO{\dsum }%
%BeginExpansion
{\displaystyle\sum}
%EndExpansion
}\left\vert A_{w}\left(  w,N\left(  am+bn\right)  \right)  \right\vert
^{2}\right\}  ^{1/2}.
\end{align*}
All the other terms in $R\left(  w,m,n\right)  $ can be estimated in a similar
way. We remark about the notation that the $w$'s in part (1) and (2) of the
above lemma are arbitrary, not necessarily equal, and not equal to the $w$ in
the theorem. This completes the proof of Theorem \ref{Theorem 3}.
\end{proof}

\bigskip

\section{Quadrature formulas for polygons}

In this section, we give quadrature formulas whose nodes are all of the integer points.
We return to our use of compact vector notation $x  \in \mathbb{R}^{2}$ and
$n \in \mathbb{Z}^{2}$. The constants $\left\{  \delta\left(  j\right)
\right\}  $ in Theorem \ref{Theorem 3} can be computed explicitly in a finite
number of steps, but they are composed of many pieces, and the final result is
not as clean as it is in $\mathbb R^1$. However, disregarding these terms in the
asymptotic expansion, one recognizes an analog of the trapezoidal rule for
approximating integrals.

\bigskip

\begin{corollary}
\label{Corollary 1}If $P$ is an open integer polygon, and if $g\left(
x\right)  $ is a smooth function, then for every positive integer $N$,
\begin{gather*}
\left\vert
%TCIMACRO{\dint _{P}}%
%BeginExpansion
{\displaystyle\int_{P}}
%EndExpansion
g\left(  x\right)  dx-\left(  \dfrac{1}{N^{2}}%
%TCIMACRO{\dsum _{N^{-1}n\in P}}%
%BeginExpansion
{\displaystyle\sum_{N^{-1}n\in P}}
%EndExpansion
g\left(  N^{-1}n\right)  +\dfrac{1}{2N^{2}}%
%TCIMACRO{\dsum _{N^{-1}n\in\partial P}}%
%BeginExpansion
{\displaystyle\sum_{N^{-1}n\in\partial P}}
%EndExpansion
g\left(  N^{-1}n\right)  \right)  \right\vert \\
\leq\dfrac{C}{N^{2}}.
\end{gather*}

\end{corollary}

\begin{proof}
By Theorem \ref{Theorem 3}, it suffices to observe that on the boundary of the polygon we have $\omega
_{P}\left(  N^{-1}n\right)  =1/2$, except at the vertices, where the
contribution is of the order of $N^{-2}$, hence negligible.
\end{proof}

\bigskip

In the above corollary there are no weights $\omega_{P}\left(  P_{j}\right)  $
at the vertices $P_{j}$. When $P$ is an integer polygon, then $\omega
_{P}\left(  P_{j}\right)  =\left(  2\pi\right)  ^{-1}\arctan\left(
a/b\right)  $, with $a$ and $b$ suitable integers, and it can be proved that
either this weight is an integer multiple of $1/8$, or it is irrational. See
Corollary 3.12 in \cite{N}. Hence Corollary \ref{Corollary 1} is, so to speak,
a rational approximation of Theorem \ref{Theorem 3}. Anyhow, by an elementary
trick suggested by Huygens and Newton, one can accelerate the convergence of
the Riemann sums and obtain a result which is better than the one above.
Huygens in \textit{"De circuli magnitudine inventa"} (1654) proved the
following:

\begin{quote} 
 \textit{The circumference of a circle is larger than the perimeter
of an inscribed equilateral polygon plus one third of the difference between
the perimeter of this polygon and the perimeter of an inscribed polygon with
half number of sides.}
\end{quote}

%\bigskip
\noindent 
The latter statement reduces to the trigonometric inequality
\[
\pi>2N\sin\left(  \pi/2N\right)  +\dfrac{1}{3}\left(  2N\sin\left(
\pi/2N\right)  -N\sin\left(  \pi/N\right)  \right)  .
\]
Observe that this inequality gives a better approximation to $\pi$ than the
inequality $\pi>2N\sin\left(  \pi/2N\right)  $ used by Archimedes in the
\textit{"Dimensio circuli". } Newton in his correspondence with Leibniz through
Oldenburg, (\textit{"Epistola Prior"} 13/6/1676) explained the theorem of
Huygens in term of the power series expansion of the sine function:%
\begin{gather*}
\sin\left(  x\right)  =x-x^{3}/6+x^{5}/120-...,\\
\dfrac{4}{3}\dfrac{\sin\left(  x\right)  }{x}-\dfrac{1}{3}\dfrac{\sin\left(
2x\right)  }{2x}=1-x^{4}/30+...
\end{gather*}
By applying this trick to our Riemann sums one can accelerate their
convergence from $N^{-2}$ to $N^{-4}$, or $N^{-6}$, or in principle to an
arbitrary speed.

\bigskip

\begin{corollary}
\textbf{\label{Corollary 2}}If $P$ is an integer polygon, and if $g\left(
x\right)  $ is a smooth function, set
\[
S\left(  N\right)  =N^{-2}%
%TCIMACRO{\dsum _{n\in\mathbb{Z}^{2}}}%
%BeginExpansion
{\displaystyle\sum_{n\in\mathbb{Z}^{2}}}
%EndExpansion
\omega_{P}\left(  N^{-1}n\right)  g\left(  N^{-1}n\right)  .
\]
Then there exists $C$ such that for every positive integer $N$,
\[
\left\vert
%TCIMACRO{\dint _{P}}%
%BeginExpansion
{\displaystyle\int_{P}}
%EndExpansion
g\left(  x\right)  dx-\left(  -\dfrac{1}{3}S\left(  N\right)  +\dfrac{4}%
{3}S\left(  2N\right)  \right)  \right\vert \leq\dfrac{C}{N^{4}}.
\]
Similarly, there exists $C$ such that for every positive integer $N$,
\[
\left\vert
%TCIMACRO{\dint _{P}}%
%BeginExpansion
{\displaystyle\int_{P}}
%EndExpansion
g\left(  x\right)  dx-\left(  \dfrac{1}{45}S\left(  N\right)  -\dfrac{20}%
{45}S\left(  2N\right)  +\dfrac{64}{45}S\left(  4N\right)  \right)
\right\vert \leq\dfrac{C}{N^{6}}.
\]
And so on...
\end{corollary}

\begin{proof}
By the theorem, there exist $\alpha$, $\beta$, $\gamma$,... such that
\[
S\left(  N\right)  =\alpha+\dfrac{\beta}{N^{2}}+\dfrac{\gamma}{N^{4}}+...
\]
Multiply $S\left(  N\right)  $ and $S\left(  2N\right)  $ by some constants
$x$ and $y$, and add,
\[
xS\left(  N\right)  +yS\left(  2N\right)  =\left(  x+y\right)  \alpha+\left(
x+y/4\right)  \dfrac{\beta}{N^{2}}+\left(  x+y/16\right)  \dfrac{\gamma}%
{N^{4}}+...
\]
Then, if $x=-1/3$ and $y=4/3$,
\[
-\dfrac{1}{3}S\left(  N\right)  +\dfrac{4}{3}S\left(  2N\right)
=\alpha-\dfrac{\gamma}{4N^{4}}+...
\]
Similarly, with a suitable linear combination $xS\left(  N\right)  +yS\left(
2N\right)  +zS\left(  4N\right)  $ one obtains an approximation of $\alpha$ of
order $N^{-6}$, and  this process may be continued.  Observe that the matrix associated to the linear system is a Vandermonde matrix.
\end{proof}

\bigskip

Observe that the sampling points $N^{-1}\mathbb{Z}^{2}\cap P$ that appear in
the sum $S\left(  N\right)  $ are a subset of the sampling points in $S\left(
2N\right)  $, and these are a subset of the ones in $S\left(  4N\right)  $,...
In the computation of $xS\left(  N\right)  +yS\left(  2N\right)  +zS\left(
4N\right)  +...$ one can collect equal sampling points, and multiply by
suitable weights. In particular the computation of $xS\left(  N\right)
+yS\left(  2N\right)  +zS\left(  4N\right)  +...$ has the same complexity of
the computation of the last summand $S\left(  2^{n}N\right)  $. As an explicit
example, observe that

\begin{align*}
-\dfrac{1}{3}S\left(  N\right)  +\dfrac{4}{3}S\left(  2N\right) & =  
-\dfrac{1}{3}N^{-2}%
%TCIMACRO{\dsum _{n\in\mathbb{Z}^{2}}}%
%BeginExpansion
{\displaystyle\sum_{n\in\mathbb{Z}^{2}}}
%EndExpansion
\omega_{P}\left(  N^{-1}n\right)  g\left(  N^{-1}n\right) \\
& \ \  \ \ +\dfrac{4}{3}\left(  2N\right)  ^{-2}%
%TCIMACRO{\dsum _{n\in\mathbb{Z}^{2}}}%
%BeginExpansion
{\displaystyle\sum_{n\in\mathbb{Z}^{2}}}
%EndExpansion
\omega_{P}\left(  \left(  2N\right)  ^{-1}n\right)  g\left(  \left(
2N\right)  ^{-1}n\right) \\
&  =\dfrac{1}{3}N^{-2}%
%TCIMACRO{\dsum _{n\in\mathbb{Z}^{2}-2\mathbb{Z}^{2}}}%
%BeginExpansion
{\displaystyle\sum_{n\in\mathbb{Z}^{2}-2\mathbb{Z}^{2}}}
%EndExpansion
\omega_{P}\left(  \left(  2N\right)  ^{-1}n\right)  g\left(  \left(
2N\right)  ^{-1}n\right)  .
\end{align*}
The vertices of the polygon do not appear in the last sum, since if $\left(
2N\right)  ^{-1}n$ is an integer vertex, then $n\in2\mathbb{Z}^{2}$. Hence in
the last sum $\omega_{P}\left(  \left(  2N\right)  ^{-1}n\right)  $ takes only
the values $0$, $1/2$, $1$,
\begin{align*}
&  -\dfrac{1}{3}S\left(  N\right)  +\dfrac{4}{3}S\left(  2N\right) \\
=  &  \dfrac{1}{3N^{2}}%
%TCIMACRO{\dsum _{\substack{n\in\mathbb{Z}^{2}\setminus2\mathbb{Z}^{2}\\\left(
%2N\right)  ^{-1}n\in\overset{o}{P}}}}%
%BeginExpansion
{\displaystyle\sum_{\substack{n\in\mathbb{Z}^{2}\setminus2\mathbb{Z}%
^{2}\\\left(  2N\right)  ^{-1}n\in\overset{o}{P}}}}
%EndExpansion
g\left(  \left(  2N\right)  ^{-1}n\right)  +\dfrac{1}{6N^{2}}%
%TCIMACRO{\dsum _{\substack{n\in\mathbb{Z}^{2}\setminus2\mathbb{Z}^{2}\\\left(
%2N\right)  ^{-1}n\in\partial P}}}%
%BeginExpansion
{\displaystyle\sum_{\substack{n\in\mathbb{Z}^{2}\setminus2\mathbb{Z}%
^{2}\\\left(  2N\right)  ^{-1}n\in\partial P}}}
%EndExpansion
g\left(  \left(  2N\right)  ^{-1}n\right)  .
\end{align*}

The statement of the corollary seems paradoxical. One trows away one fourth of
the grid $\left(  2N\right)  ^{-1}\mathbb{Z}^{2}\cap P$, and the order of
approximation increases.

This trick of Huygens and Newton works in every dimension. In particular, in
dimension one the weighted sum $-\dfrac{1}{3}S\left(  N\right)  +\dfrac{4}%
{3}S\left(  2N\right)  $ reduces to the Kepler, or Cavalieri, or Simpson
quadrature rule,
\begin{align*}
&  -\dfrac{1}{3}\left[  \dfrac{1}{2N}g\left(  a\right)  +\dfrac{1}{N}g\left(
a+\dfrac{1}{N}\right)  +\dfrac{1}{N}g\left(  a+\dfrac{2}{N}\right)
+...\right] \\
&  +\dfrac{4}{3}\left[  \dfrac{1}{4N}g\left(  a\right)  +\dfrac{1}{2N}g\left(
a+\dfrac{1}{2N}\right)  +\dfrac{1}{2N}g\left(  a+\dfrac{2}{2N}\right)
+\dfrac{1}{2N}g\left(  a+\dfrac{3}{2N}\right)  +...\right] \\
=  &  \dfrac{1}{6N}\left[  g\left(  a\right)  +4g\left(  a+\dfrac{1}%
{2N}\right)  +2g\left(  a+\dfrac{2}{2N}\right)  \right. \\
&  ~~~~~~~~~~~~~~\left.  +4g\left(  a+\dfrac{3}{2N}\right)  +2g\left(
a+\dfrac{4}{2N}\right)  +...\right]
\end{align*}

\textit{"What has been will be again, what has been done will be done again;
there is nothing new under the sun"}.

\bigskip

Let us go back to the theorem. When $g\left(  x\right)  $ is a polynomial only
a finite number of coefficients $\left\{  \delta\left(  h\right)  \right\}  $
are nonzero, and the asymptotic formula becomes exact.

\bigskip

\begin{corollary}
\label{Corollary 3}
For a homogeneous polynomial $g(x)$ of degree $\alpha$, we have:
\[%
%TCIMACRO{\dsum _{n\in\mathbb{Z}^{2}}}%
%BeginExpansion
{\displaystyle\sum_{n\in\mathbb{Z}^{2}}}
%EndExpansion
\omega_{NP}\left(  n\right)  g\left(  n\right)  =N^{\alpha+2}%
%TCIMACRO{\dint _{P}}%
%BeginExpansion
{\displaystyle\int_{P}}
%EndExpansion
g\left(  x\right)  dx+N^{\alpha+2}%
%TCIMACRO{\dsum _{j=1}^{w}}%
%BeginExpansion
{\displaystyle\sum_{j=1}^{w}}
%EndExpansion
\dfrac{\delta\left(  j\right)  }{N^{2j}}.
\]

\end{corollary}

\begin{proof}
In Theorem \ref{Theorem 3}, the reminder $R\left(  w,N\right)  $ vanishes
provided that all derivatives of order $2w+1$ vanish. Hence, for every
$w>\left(  \alpha-1\right)  /2$,
\[
N^{-2}%
%TCIMACRO{\dsum _{n\in\mathbb{Z}^{2}}}%
%BeginExpansion
{\displaystyle\sum_{n\in\mathbb{Z}^{2}}}
%EndExpansion
\omega_{P}\left(  N^{-1}n\right)  g\left(  N^{-1}n\right)  =%
%TCIMACRO{\dint _{P}}%
%BeginExpansion
{\displaystyle\int_{P}}
%EndExpansion
g\left(  x\right)  dx+%
%TCIMACRO{\dsum _{j=1}^{w}}%
%BeginExpansion
{\displaystyle\sum_{j=1}^{w}}
%EndExpansion
\dfrac{\delta\left(  j\right)  }{N^{2j}}.
\]
Because $g\left(  x\right)  $ is homogeneous of degree $\alpha$, we have
\[
\omega_{P}\left(  N^{-1}n\right)  g\left(  N^{-1}n\right)  =N^{-\alpha}%
\omega_{NP}\left(  n\right)  g\left(  n\right).
\]
\end{proof}

\bigskip
The case $g(x) =1$ of Corollary \ref{Corollary 3}  is the celebrated MacDonald solid-angle polynomial in two dimensions, from which the classical Pick's formula follows easily.

\bigskip

\begin{corollary} (Pick's formula)   \label{Corollary 4}
If $P$ is a simply connected integer polygon in the plane,
with $I$ interior integer points, $B$ boundary integer points, and area
$\left\vert P\right\vert $, then
\[
\left\vert P\right\vert =I+\frac{B}{2}-1.
\]

\end{corollary}

\begin{proof}
By Corollary \ref{Corollary 3} with $g(x)=1$ and  $N=1$,
\[%
%TCIMACRO{\dint _{P}}%
%BeginExpansion
{\displaystyle\int_{P}}
%EndExpansion
dx=%
%TCIMACRO{\dsum _{n\in\mathbb{Z}^{2}}}%
%BeginExpansion
{\displaystyle\sum_{n\in\mathbb{Z}^{2}}}
%EndExpansion
\omega_{P}\left(  n\right)  .
\]
If $n$ is an integer point inside $P$, then $\omega_{P}\left(  n\right)  =1$.
If $n$ is on a side but it is not a vertex, then $\omega_{P}\left(  n\right)
=1/2$. If $P$ has $r$ vertices, then
\[%
%TCIMACRO{\dsum _{n\ \text{vertex\ of}\ P}}%
%BeginExpansion
{\displaystyle\sum_{n\ \text{vertex\ of}\ P}}
%EndExpansion
\omega_{P}\left(  n\right)  =\dfrac{\left(  r-2\right)  \pi}{2\pi}=\dfrac
{r}{2}-1.
\]
It follows that
\[%
%TCIMACRO{\dsum _{n\in\mathbb{Z}^{2}}}%
%BeginExpansion
{\displaystyle\sum_{n\in\mathbb{Z}^{2}}}
%EndExpansion
\omega_{P}\left(  n\right)  =I+\frac{B}{2}-1.
\]
\end{proof}

\bigskip
In \cite{B-C-R-T}, we presented a harmonic analysis proof of Pick's theorem, together with one possible conjecture for an extension to higher dimensions.   See \cite{P} for a very recent counterexample to this conjecture.

\section{Appendix: A numerical example}

Let us test Corollary \ref{Corollary 1} and Corollary \ref{Corollary 2} on an
explicit example. Take $P=\left\{  y>x/2,\ y<3-x,\ y<2x\right\}  $, $g\left(
x,y\right)  =x^{2}y^{3}$, and $N=2$, that is a grid of side $1/4$. The exact
value of the integral is
\[%
%TCIMACRO{\dint }%
%BeginExpansion
{\displaystyle\int}
%EndExpansion%
%TCIMACRO{\dint _{\left\{  y>x/2,\ y<2-x,\ y<2x\right\}  }}%
%BeginExpansion
{\displaystyle\int_{\left\{  y>x/2,\ y<2-x,\ y<2x\right\}  }}
%EndExpansion
x^{2}y^{3}dxdy=\dfrac{423}{140}=3.021...
\]
The Riemann sum in Corollary \ref{Corollary 1} runs over $31$ sampling points,
and it is a rough approximation to the integral:
\begin{align*}
&  \dfrac{1}{16}%
%TCIMACRO{\dsum _{4^{-1}\left(  m,n\right)  \in P}}%
%BeginExpansion
{\displaystyle\sum_{4^{-1}\left(  m,n\right)  \in P}}
%EndExpansion
\left(  \dfrac{m}{4}\right)  ^{2}\left(  \dfrac{n}{4}\right)  ^{3}+\dfrac
{1}{32}%
%TCIMACRO{\dsum _{4^{-1}\left(  m,n\right)  \in\partial P}}%
%BeginExpansion
{\displaystyle\sum_{4^{-1}\left(  m,n\right)  \in\partial P}}
%EndExpansion
\left(  \dfrac{m}{4}\right)  ^{2}\left(  \dfrac{n}{4}\right)  ^{3}\\
=  &  \dfrac{54335}{16384}=3.316...
\end{align*}
The Riemann sum in Corollary \ref{Corollary 2} runs over $21$ sampling points,
and it gives a much better approximation to the integral above:
\[
\dfrac{1}{12}%
%TCIMACRO{\dsum _{\substack{\left(  m,n\right)  \in\mathbb{Z}^{2}%
%\setminus2\mathbb{Z}^{2}\\4^{-1}\left(  m,n\right)  \in P\ }}}%
%BeginExpansion
{\displaystyle\sum_{\substack{\left(  m,n\right)  \in\mathbb{Z}^{2}%
\setminus2\mathbb{Z}^{2}\\4^{-1}\left(  m,n\right)  \in P\ }}}
%EndExpansion
\left(  \dfrac{m}{4}\right)  ^{2}\left(  \dfrac{n}{4}\right)  ^{3}+\dfrac
{1}{24}%
%TCIMACRO{\dsum _{\substack{\left(  m,n\right)  \in\mathbb{Z}^{2}%
%\setminus2\mathbb{Z}^{2}\\4^{-1}\left(  m,n\right)  \in\partial P\ }}}%
%BeginExpansion
{\displaystyle\sum_{\substack{\left(  m,n\right)  \in\mathbb{Z}^{2}%
\setminus2\mathbb{Z}^{2}\\4^{-1}\left(  m,n\right)  \in\partial P\ }}}
%EndExpansion
\left(  \dfrac{m}{4}\right)  ^{2}\left(  \dfrac{n}{4}\right)  ^{3}%
=\dfrac{37295}{12288}=3.035...
\]

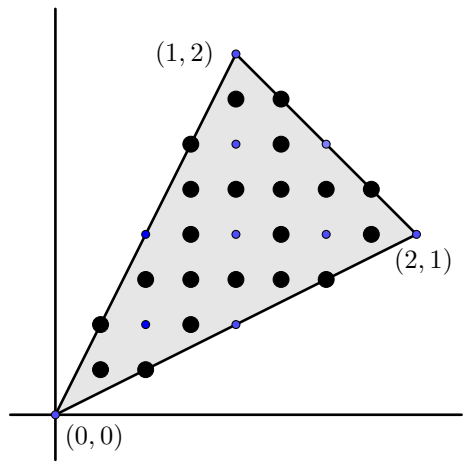
\begin{figure}[hptb]
\centering
\definecolor{xdxdff}{rgb}{0.49019607843137253,0.49019607843137253,1.}
\definecolor{qqqqff}{rgb}{0.,0.,1.}
\definecolor{ududff}{rgb}{0.30196078431372547,0.30196078431372547,1.}
\begin{tikzpicture}[line cap=round,line join=round,>=triangle 45,x=0.6cm,y=0.6cm]
\fill[line width=0.3pt,fill=black,fill opacity=0.1] (-7.,-6.) -- (1.,-2.) -- (-3.,2.) -- cycle;
\draw [line width=1.pt] (-7.,-6.)-- (1.,-2.);
\draw [line width=1.pt] (1.,-2.)-- (-3.,2.);
\draw [line width=1.pt] (-3.,2.)-- (-7.,-6.);
\draw [line width=1.pt] (-8.,-6.)-- (2.,-6.);
\draw [line width=1.pt] (-7.,-7.)-- (-7.,3.);
\draw (-7.,-6) node[anchor=north west] {$(0,0)$};
\draw (0.3,-2.1) node[anchor=north west] {$(2,1)$};
\draw (-5.,2.5) node[anchor=north west] {$(1,2)$};
\draw [fill=ududff] (-7.,-6.) circle (1.5pt);
\draw [fill=ududff] (-3.,-4.) circle (1.5pt);
\draw [fill=ududff] (1.,-2.) circle (1.5pt);
\draw [fill=ududff] (-3.,2.) circle (1.5pt);
\draw [fill=qqqqff] (-5.,-2.) circle (1.5pt);
\draw [fill=black] (-6.,-5.) circle (3.pt);
\draw [fill=black] (-6.,-4.) circle (3.pt);
\draw [fill=qqqqff] (-5.,-4.) circle (1.5pt);
\draw [fill=black] (-5.,-5.) circle (3.pt);
\draw [fill=black] (-5.,-3.) circle (3.pt);
\draw [fill=black] (-4.,-3.) circle (3.pt);
\draw [fill=black] (-4.,-4.) circle (3.pt);
\draw [fill=black] (-4.,-2.) circle (3.pt);
\draw [fill=black] (-4.,-1.) circle (3.pt);
\draw [fill=black] (-4.,0.) circle (3.pt);
\draw [fill=ududff] (-3.,0.) circle (1.5pt);
\draw [fill=black] (-3.,1.) circle (3.pt);
\draw [fill=black] (-3.,-1.) circle (3.pt);
\draw [fill=ududff] (-3.,-2.) circle (1.5pt);
\draw [fill=black] (-3.,-3.) circle (3.pt);
\draw [fill=black] (-2.,-3.) circle (3.pt);
\draw [fill=black] (-2.,-2.) circle (3.pt); \draw [fill=black] (-2.,-1.) circle (3.pt); \draw [fill=black] (-2.,0.) circle (3.pt); \draw [fill=black] (-2.,1.) circle (3.pt); \draw [fill=xdxdff] (-1.,0.) circle (1.5pt); \draw [fill=black] (-1.,-1.) circle (3.pt); \draw [fill=ududff] (-1.,-2.) circle (1.5pt); \draw [fill=black] (-1.,-3.) circle (3.pt); \draw [fill=black] (0.,-2.) circle (3.pt); \draw [fill=black] (0.,-1.) circle (3.pt); \end{tikzpicture}
\caption{Small  and large sampling points in $4^{-1}\mathbb{Z}^{2}$ in Corollary
\ref{Corollary 1}, and large sampling points in $4^{-1}\mathbb{Z}^{2}%
-2^{-1}\mathbb{Z}^{2}$ in Corollary \ref{Corollary 2}.}%
\end{figure}
%EndExpansion

\newpage
\end{document}